\documentclass{article}
\usepackage[paper=a4paper,left=30mm,right=30mm,top=25mm,bottom=25mm]{geometry}
\usepackage{lmodern}
\usepackage{dsfont}  
\usepackage{amsthm} 
\usepackage{amssymb} 
\usepackage{mathtools}
\usepackage{enumitem} 
\usepackage{booktabs} 
\usepackage{hyperref}
\usepackage{threeparttable}
\newtheorem{thm}{Theorem}[section]
\newtheorem{theorem}[thm]{Theorem}
\newtheorem{lemma}[thm]{Lemma}
\newtheorem{corollary}[thm]{Corollary}
\newtheorem{proposition}[thm]{Proposition}
\theoremstyle{definition}

\newtheorem{example}[thm]{Example}

\newcommand{\CC}{\mathcal{C}}  
\newcommand{\II}{\mathcal{I}}  
\newcommand{\OO}{\mathcal{O}}  
\newcommand{\Prob}{\mathbb{P}}  
\DeclareMathOperator{\Unif}{Unif}  
\DeclareMathOperator{\tr}{tr}  
\newcommand{\overbar}[1]{\mkern 1.5mu\overline{\mkern-1.5mu#1\mkern-1.5mu}\mkern 1.5mu}
\def\1{\mathds{1}} 

\def\abs#1{\left\lvert#1\right\rvert} 
\newcommand{\de}{\mathrm{\,d}}
\newcommand\stackrelmt[2]{\stackrel{\mathclap{\text{#1}}}{#2}}
\newcommand{\bsX}{\boldsymbol{X}}
\newcommand{\bsY}{\boldsymbol{Y}}

\author{Marcus Rockel}
\title{Measures of association for approximating copulas}

\begin{document}
\maketitle
\vspace{-5ex}

\begin{center}
	\small\textit{
		Department of Quantitative Finance,\\
		Institute for Economics, University of Freiburg,\\
		Rempartstr. 16,	79098 Freiburg, Germany,\\
        marcus.rockel@finance.uni-freiburg.de \\[2mm]
	}
\end{center}
\begin{abstract}
    This paper studies closed-form expressions for multiple association measures of copulas commonly used for approximation purposes, including Bernstein, shuffle--of--min, checkerboard and check--min copulas.
    In particular, closed-form expressions are provided for the recently popularized Chatterjee's $\xi$, which quantifies the dependence between two random variables.
    Given an absolutely continuous bivariate copula $C$ with TP$_2$ density and approximating $n\times n$-checkerboard copula $C_n$, we show that $\xi(C_n) \le \xi(C)$ with $\xi(C_n) \to \xi(C)$ as $n\to\infty$.
\end{abstract}
\vspace{0ex}

\section{Introduction}
\label{sec:introduction}

Measures of association—most prominently Spearman's rho and Kendall's tau—are fundamental tools for studying statistical dependence, and \cite{chatterjee2020} and \cite{chatterjee2021} popularized Chatterjee's xi to capture arbitrary functional dependence.
Closed--form expressions for these statistics, however, exist only for a handful of copula families (see \cite[Table~6]{ansari2024dependence}); in general one must resort to numerical or sampling–based procedures.
Near the boundaries of the unit square, such procedures often become unstable because they require evaluating (conditional) distribution functions where numerical precision can quickly become poor.
We therefore study measures of association for several popular approximation families, which we introduce in Section~\ref{sec:basic_concepts_and_notation}.

Bernstein and checkerboard constructions, in particular, have a rich history and broad practical use, see, e.g., \cite{cottin2014bernstein,kuzmenko2020checkerboard,li1997approximation,li1998strong,sancetta2004bernstein,sukeda2023minimum}.
Closed–form formulas for Spearman's rho and Kendall's tau are already known for these families, the most elegant arguably appearing in \cite{durrleman2000copulas}.
In Section~\ref{sec:explicit_measures_of_association}, we extend this catalogue by deriving explicit formulas for Chatterjee's xi not only for Bernstein and checkerboard copulas, but also for the check--min and check--$w$ variants, whose grid cells exhibit perfect dependence.
We additionally collect complete closed–form expressions for Spearman's rho, Kendall's tau, and the tail–dependence coefficients, thereby unifying and extending earlier results.

In Section~\ref{sec:checkerboard_bounds}, we focus on bounding Chatterjee's xi via checkerboard approximations.
Combining Proposition~\ref{lem:checkerboard_measures_of_association} with Theorem~\ref{thm:xi_bounds}, we establish the inequality
\begin{equation}\label{eq:xi_bound}
    \xi(C^\Delta_\Pi)
    = \frac{6m}{n}\tr\left(\Delta^{\top}\Delta M_{\xi}\right) - 2
    \le \xi(C)
\end{equation}
for absolutely continuous copulas $C$ with TP$_2$ density, where $\Delta$ is the $m\times n$-matrix of copula masses on an equally spaced grid, $M_{\xi}$ is defined by $M_{\xi} = TT^\top + T^\top + \tfrac13 I_{n}$ with $T_{i,j} = \1_{\{i<j\}}$, and $C^\Delta_\Pi$ is the checkerboard copula associated with $\Delta$.
We show that the MTP$_2$ assumption is critical, and that an upper bound via check--min copulas is not even possible under this positive-dependence assumption.
Replacing~$C$ by its associated checkerboard copula thus furnishes a practical estimator of $\xi$ from an analytic copula, but also from empirical data.
In Theorem~\ref{thm:xi_estimator} we prove that the resulting sample-based estimator converges to the true value of $\xi(C)$ and is computable in time $\mathcal{O}(n\log n)$.
Checkerboard estimators for dependence measures have been recently investigated in a broader setting in \cite[Section~4]{Ansari-Fuchs-2023}, but the explicit formulas derived here allow a finer–grained analysis and faster finite–sample performance.
We conclude with an empirical comparison between our estimator and the classical nearest-neighbor-based estimater of Azadkia and Chatterjee in \cite{chatterjee2021}.

\section{Preliminaries}
\label{sec:basic_concepts_and_notation} 

In this section, we introduce the basic concepts and notation used that are required to formulate the main results of this paper.
First, we introduce the fundamental concept of a copula, before focusing on the specific types of copulas that are of interest in this paper.
Finally, we introduce the studied measures of association.

\subsection{Copulas}
\label{subsec:copulas}

A \emph{bivariate copula} is a function \(C\colon [0,1]^2 \to [0,1]\) that is \emph{grounded} (i.e., \(C(u,0)=C(0,v)=0\) for all \(u,v\in[0,1]\)), \emph{\(2\)-increasing} (meaning that for every \(0\le u_1\le v_1\le 1\) and \(0\le u_2\le v_2\le 1\) it holds that \(C(v_1,v_2)-C(u_1,v_2)-C(v_1,u_2)+C(u_1,u_2)\ge 0\)), and has uniform marginals (so that \(C(u,1)=u\) and \(C(1,v)=v\) for all \(u,v\in[0,1]\)).
Sklar's theorem (see, e.g., \cite[Theorem~2.3.3]{Nelsen-2006}) states that for any bivariate distribution function \(F\) with univariate marginals \(F_1\) and \(F_2\), there exists a copula \(C\) such that
\begin{equation}\label{eq:sklar}
F(x_1,x_2)=C\left(F_1(x_1),F_2(x_2)\right)\quad \text{for all } (x_1,x_2)\in\mathbb{R}^2,
\end{equation}
and \(C\) is uniquely determined on \(\operatorname{Ran}(F_1) \times \operatorname{Ran}(F_2)\). Conversely, if \(C\) is any copula and \(F_1\), \(F_2\) are univariate distribution functions, then the function defined by \eqref{eq:sklar} is a bivariate distribution function.

Denote by \(\CC_2\) the collection of all bivariate copulas. Important examples include the \emph{independence copula} $\Pi(u,v)=uv$, the \emph{upper Fr\'echet bound} $M(u,v)=\min\{u,v\}$ and the \emph{lower Fr\'echet bound} $W(u,v)=\max\{u+v-1,0\}$.
Classically, if $(X,Y)\sim C$, the upper and lower Fr\'echet bounds represent the extreme cases of dependence with perfect co- and countermonotonicity, respectively, whilst the independence copula represents the case of no dependence at all between $X$ and $Y$.
Furthermore, for any \(C\in\CC_2\), it holds that $W\le C\le M$ pointwise on $[0,1]^2$, see standard references such as \cite{Nelsen-2006} or \cite{Durante-2016}.

\subsection{Bernstein copulas}
\label{subsec:bernstein_copulas}

Bernstein copulas were first considered by Li et al. \cite{li1998strong} and then considered by Sancetta and Satchell \cite{sancetta2004bernstein} as flexible, computable copulas for approximating dependence structures.
Let $C$ be a given bivariate copula and let $D$ be an $m\times n$-matrix defined by
\begin{align}\label{eq:matrix_copula}
    D_{i,j} = C\left(\frac{i}{m}, \frac{j}{n}\right) 
\end{align}
for $1\leq i\leq m$ and $1\leq j\leq n$.
We refer to $D$ as the $m\times n$-\emph{grid copula matrix associated with $C$} and generally call $D$ a \emph{grid copula matrix} if there exists a copula such that \eqref{eq:matrix_copula} holds for all entries of the matrix.
Next, let $B_{i,m}(u)$ denote the \emph{Bernstein basis polynomial} of degree $m$, defined as
\[
    B_{i,m}(u) = \binom{m}{i}u^i(1-u)^{m-i},
    \quad \text{for } 0\le i\le m,u\in[0,1].
\]
Then, the \emph{Bernstein copula} associated with the grid copula matrix $D$ is defined as
\begin{align}\label{eq:bernstein_copula}
    C^D_B(u,v) = \sum_{i=1}^m \sum_{j=1}^n D_{i,j}B_{i,m}(u)B_{j,n}(v), \quad \text{for } (u,v)\in[0,1]^2.
\end{align}
This function $C^D_B$ is indeed a copula, as shown in \cite[Theorem 1]{sancetta2004bernstein} (see also \cite[Theorem~2.2]{cottin2014bernstein}).
An illustration for a Bernstein approximation of the Clayton copula is given in Figure~\ref{fig:clayton_pdf_and_bernstein}.

\begin{figure}[htb!]
    \centering
    \includegraphics[width=0.4\textwidth]{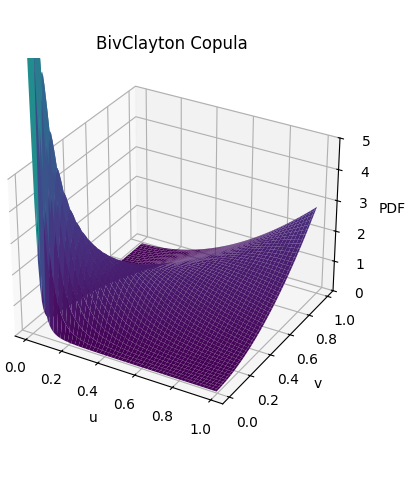}
    \includegraphics[width=0.4\textwidth]{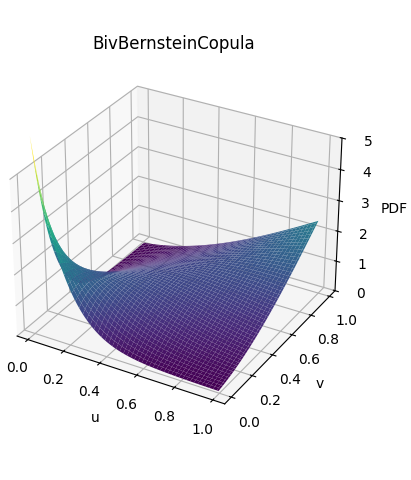}
    \caption{
        Left: Probability density function of the Clayton copula with parameter $\theta=2$.
        Right: Probability density function of the $20\times 20$-Bernstein copula associated with the Clayton copula with parameter $\theta=2$.
        The Bernstein copula is a smooth approximation of the Clayton copula with bounded density.
    }
    \label{fig:clayton_pdf_and_bernstein}
\end{figure}

A key feature of the Bernstein copula $C^D_B$ is that it is a polynomial in both $u$ (of degree $m$) and $v$ (of degree $n$), which ensures the resulting copula is smooth.
The parameters $m$ and $n$ determine the degree of the polynomial and thus control the trade-off between the smoothness of the approximation and its ability to capture fine details of the underlying dependence structure represented by $D$.
If one considers a sequence of grid copula matrices $D_{m,n}$ associated with $C$ and lets $m \wedge n \to \infty$, the Bernstein copula $C^{D_{m,n}}_B$ converges uniformly to $C$ \cite[Corollary 3.1]{cottin2014bernstein}. Notably, Li et al. \cite{li1998strong} established a very strong form of convergence for these approximations based on partitions of unity, which is stronger than the standard uniform convergence of copulas, as it relates to the convergence of the underlying measures.

\subsection{Shuffle--of--min copulas}
\label{subsec:shuffle_of_min_copulas}

The \emph{shuffle--of--min} construction, introduced by Mikusi\'nski, Sherwood and Taylor in \cite{mikusinski1992shuffles} (see also \cite{Nelsen-2006}), produces a rich family of singular copulas that are dense in~\(\CC_2\) with respect to the supremum norm $\|C\|_\infty = \sup_{(u,v)\in[0,1]^2} |C(u,v)|$.
Fix an integer $n\ge 1$ and partition the unit interval into equal sub–intervals $I_k=[\tfrac{k-1}{n},\tfrac{k}{n}]$ for $k=1,\dots,n$.
Denote by $\mathfrak S_n$ the set of all permutations of $\{1,\dots,n\}$ and let $\pi\in\mathfrak S_n$ be a permutation.

\begin{figure}[htbp!]
    \centering
    \includegraphics[width=0.6\textwidth]{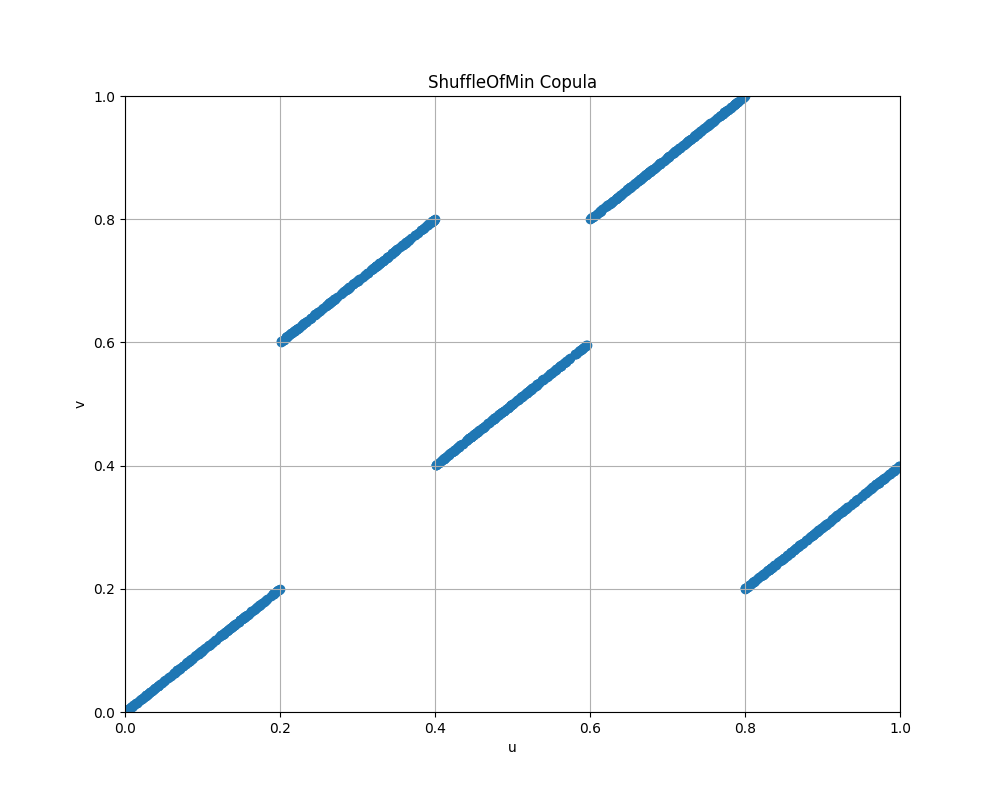}
    \caption{
        Scatter plot of $1000$ samples from the straight shuffle--of--min copula with order $n=5$ and shuffle permutation $\pi=(1,4,3,5,2)$.
    }
    \label{fig:shuffle_of_min_example}
\end{figure}

The \emph{straight shuffle--of--min copula associated with the permutation $\pi$}, denoted $C_\pi$, redistributes the probability mass of the comonotonic copula $M(u,v)=\min\{u,v\}$ equally along the $n$ diagonal line segments
\[
    \left\{(u,v)\in I_k\times I_{\pi(k)} : v=u-\tfrac{k-\pi(k)}{n}\right\},\qquad k=1,\dots,n,
\]
so that each segment carries mass $1/n$.
Equivalently, $C_\pi$ is the distribution of $(U,V)$ where $U\sim\mathrm U(0,1)$ and, conditional on $U\in I_k$, one sets $V=U-\tfrac{k-\pi(k)}{n}$.  We call $n$ the \emph{order} of the shuffle and $\pi$ its \emph{shuffle permutation}.  

More general shuffles allow unequal strip widths $p_1,\dots,p_n>0$ with $\sum p_k=1$ and/or segment reflections, but in this paper we restrict to equal--width \emph{straight} shuffles, because they are already dense in $\CC_2$ and admit closed–form formulas for the concordance measures considered below.
Sample data obtained from a simple shuffle--of--min copula is shown in Figure~\ref{fig:shuffle_of_min_example}.

\subsection{Checkerboard, check--min and check--w copulas}
\label{subsec:delta_and_checkerboard_copulas}

Let $\Delta$ be an $m\times n$-matrix.
We say that $\Delta$ is a \emph{checkerboard matrix} if all entries are nonnegative and for all $1\leq i \leq m$ and $1\leq j \leq n$ it holds that
\[
    \sum_{k=1}^{m} \Delta_{k,j} = \frac1n,\quad \sum_{l=1}^{n} \Delta_{i,l} = \frac1m
.\]
Next, divide the interval \([0,1]\) into $m$ and $n$ equal parts, respectively, and let
\begin{align}\label{eq:checkerboard_rectangles}
    \II_{i, j}
    \coloneq \left[\frac{i-1}{m}, \frac{i}{m}\right) \times \left[\frac{j-1}{n}, \frac{j}{n}\right)
.\end{align}
If $C$ is a copula and
\begin{align}\label{eq:delta_copula}
    \Prob[(X,Y) \in \II_{i,j}] = \Delta_{i,j}
\end{align}
for a random vector $(X,Y)\sim C$, we say that $\Delta$ is a \textit{checkerboard matrix associated with} $C$.
If $C$ has a constant density within each rectangle $\II_{i,j}$, then $C$ is called a \emph{checkerboard copula} and we write $C=C^{\Delta}_{\Pi}$.
The copula is explicitly given by
\begin{align}\label{eq:checkerboard_copula}
    C^{\Delta}_{\Pi}(x,y) = mn\sum_{i=1}^{m} \sum_{j=1}^{n} \Delta_{i,j} \int_{0} ^x\int_0^y \1_{\II_{i,j}}(u,v) \de v \de u
\end{align}
where $\1_{\II_{i,j}}$ is the indicator function of the rectangle $\II_{i,j}$.
$C^{\Delta}_{\Pi}$ is indeed a copula for any checkerboard matrix $\Delta$, see \cite[Section 2]{kuzmenko2020checkerboard} or, in the square case, \cite[Theorem 2.2]{li1998strong}.
Furthermore, as a simple consequence of the density being constant on each $\II_{i,j}$, it holds that
\begin{align}\label{eq:checkerboard_copula_2}
    \Prob\left[(X,Y)\leq (x,y)~\middle|~(X,Y)\in\II_{i,j}\right]
    = \Prob\left[X\leq x~\middle|~(X,Y)\in\II_{i,j}\right] \Prob\left[Y\leq y~\middle|~(X,Y)\in\II_{i,j}\right]
.\end{align}
From \eqref{eq:checkerboard_copula_2}, one obtains the following expression for the copula $C^{\Delta}_{\Pi}$, which is also covered in \cite[Theorem 4.1.3]{Durante-2016}:
\begin{align}\label{frm:check_pi}
    \begin{aligned}
    C^{\Delta}_{\Pi}(u,v)
    = \sum_{k=1}^{i-1}\sum_{l=1}^{j-1} \Delta_{k,l}
      + \sum_{k=1}^{i-1}\Delta_{k,j} (nv-j+1)
    & + \sum_{l=1}^{j-1}\Delta_{i,l}(mu-i+1) \\
    & + \Delta_{i,j} (nv-j+1)(mu-i+1)
    \end{aligned}
\end{align}
for all $(u,v)\in \II_{i,j}$ and $1\leq i \leq m$, $1\leq j \leq n$.
For a given copula $C$, considering associated $n\times n$-checkerboard matrices $\Delta_n$ yields desirable convergence properties $C^{\Delta_n}_{\Pi}\rightarrow C$ as $n\to\infty$, see, e.g., \cite[Corollary 3.2]{li1998strong}.
Next to independence within rectangles as realized through $C^{\Delta}_{\Pi}$, it is also reasonable to consider, for a given $\Delta$, perfect positive dependence within each rectangle, with mass distributed uniformly along the increasing diagonal segment of the rectangle.
More precisely, for every $1\le i\le m$ and $1\le j\le n$ with $\Delta_{i,j}>0$, conditionally on $(X,Y)\in\II_{i,j}$ we require
\begin{align}\label{eq:check_min_copula_2}
    X = \frac{nY - j + i}{m}
\end{align}
almost surely, and additionally
\(
    mX-i+1 \,\mid\, (X,Y)\in\II_{i,j} \sim \Unif(0,1).
\)
If there exist a checkerboard matrix $\Delta$ and a random vector $(X,Y)\sim C$ fulfilling \eqref{eq:delta_copula} and the preceding two conditional requirements, $C$ is called a \emph{check--min} copula, and we write \(C=C^{\Delta}_{\nearrow}\).
Check--min approximations were considered in multiple applications, see, e.g., \cite{Mikusinski-2010,zheng2011approximation,durante2015typical}.
In analogy to \eqref{eq:checkerboard_copula_2}, one can equivalently write \eqref{eq:check_min_copula_2} as
\begin{align}\label{eq:check_min_copula}
    \Prob\left[(X,Y)\leq (x,y)~\middle|~(X,Y) \in \II_{i,j}\right]
    = \Prob\left[ Y \leq \frac{mx - i + j}{n} \wedge y ~\middle|~(X,Y) \in \II_{i,j}\right]
\end{align}
for all $(x,y)\in [0,1]^2$. Using the conditional uniformity of the local coordinate, a case separation shows that
\begin{align}\label{frm:check_min}
    \begin{aligned}
    C^{\Delta}_{\nearrow}(u,v)
    = \sum_{k=1}^{i-1}\sum_{l=1}^{j-1} \Delta_{k,l}
      + \sum_{k=1}^{i-1}\Delta_{k,j} (nv-j+1)
    & + \sum_{l=1}^{j-1}\Delta_{i,l}(mu-i+1) \\
    & + \Delta_{i,j} \min{\{nv-j+1,mu-i+1\}}
    .\end{aligned}
\end{align}
Similar convergence properties as for the checkerboard copula hold for check--min copulas, see \cite{Mikusinski-2010}.
Lastly, consider also the \emph{check--w copula} \(C^{\Delta}_{\searrow}\), which represents perfect negative dependence within each rectangle, with mass distributed uniformly along the decreasing diagonal segment of the rectangle.
More precisely, for every $1\le i\le m$ and $1\le j\le n$ with $\Delta_{i,j}>0$, conditionally on $(X,Y)\in\II_{i,j}$ we require
\begin{align}\label{eq:check_w_copula_2}
    X = \frac{i - 1 + j - nY}{m}
\end{align}
almost surely, and additionally
\[
    mX-i+1 \,\mid\, (X,Y)\in\II_{i,j} \sim \Unif(0,1).
\]
In particular, in analogy to \eqref{eq:check_min_copula}, one can write \eqref{eq:check_w_copula_2} equivalently as
\begin{align*}
    \Prob\left[X\leq x, Y\leq y~\middle|~(X,Y)\in \II_{i,j}\right]
    = \Prob\left[\frac{j-1+i-mx}n \leq Y \leq y~\middle|~ (X,Y)\in \II_{i,j}\right]
\end{align*}
for all $(x,y)\in [0,1]^2$, and, again using the conditional uniformity of the local coordinate, another case separation shows that
\begin{align}\label{frm:check_w}
    \begin{aligned}
    C^{\Delta}_{\searrow}(u,v)
    = \sum_{k=1}^{i-1}\sum_{l=1}^{j-1} \Delta_{k,l}
      + \sum_{k=1}^{i-1}\Delta_{k,j} (nv-j+1)
    & + \sum_{l=1}^{j-1}\Delta_{i,l}(mu-i+1) \\
    & + \Delta_{i,j} \max{\{nv-j + mu-i + 1, 0\}}
    .\end{aligned}
\end{align}

\begin{figure}[htb!]
    \centering
    \includegraphics[width=0.32\textwidth]{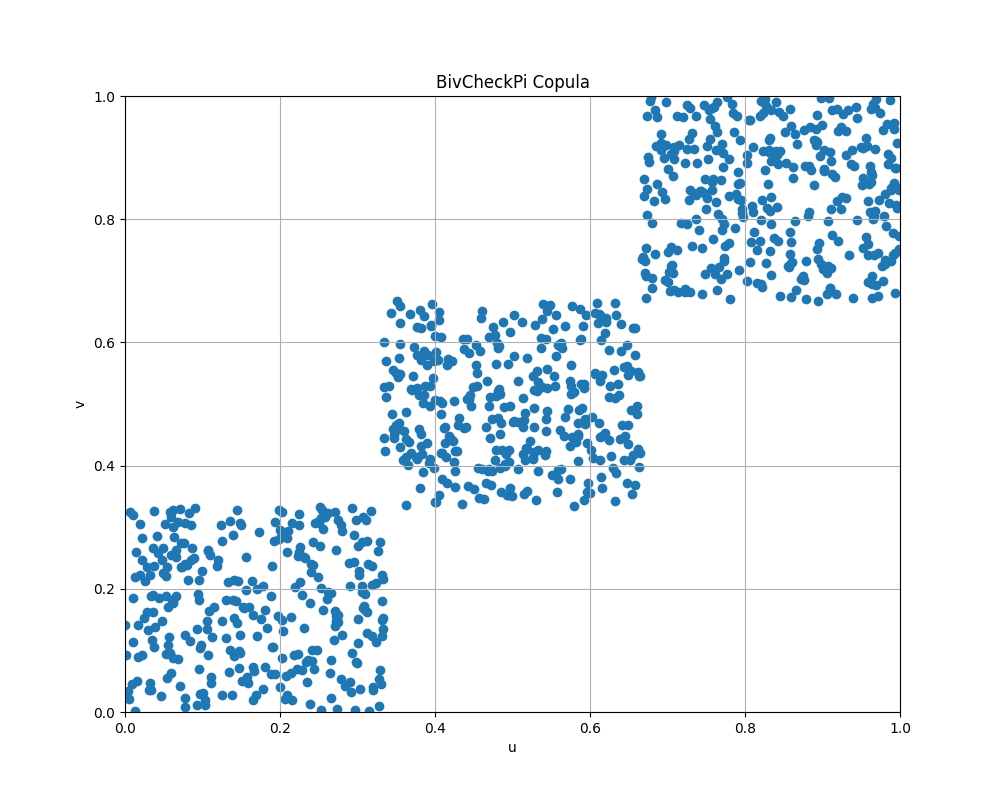}
    \includegraphics[width=0.32\textwidth]{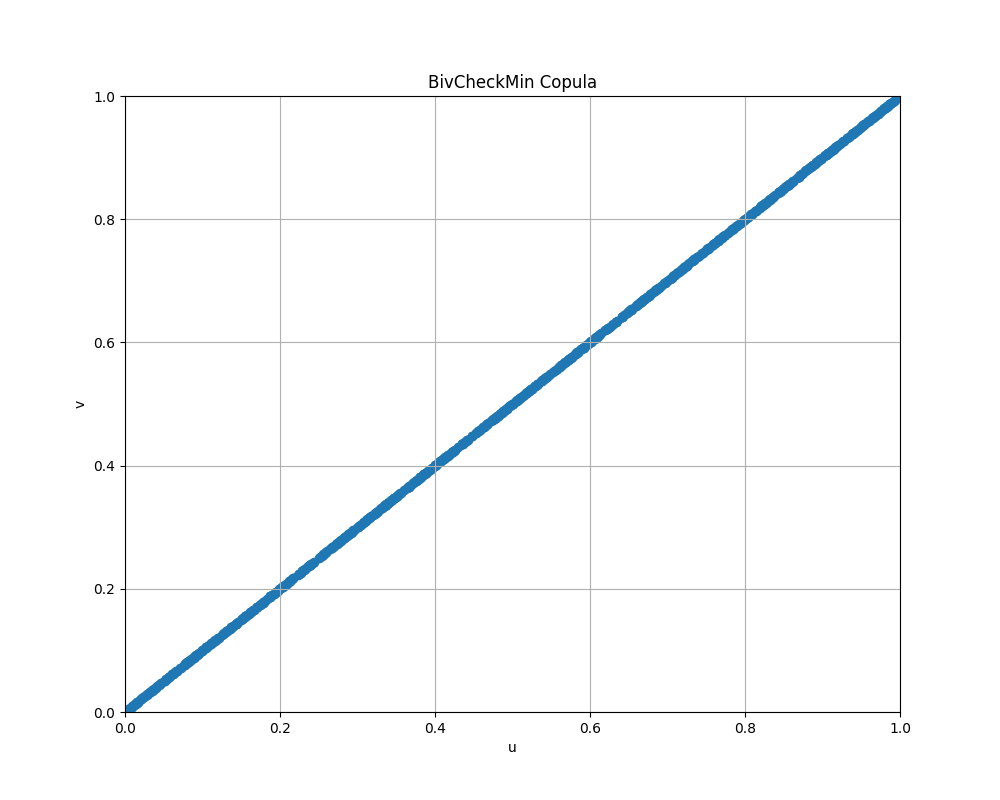}
    \includegraphics[width=0.32\textwidth]{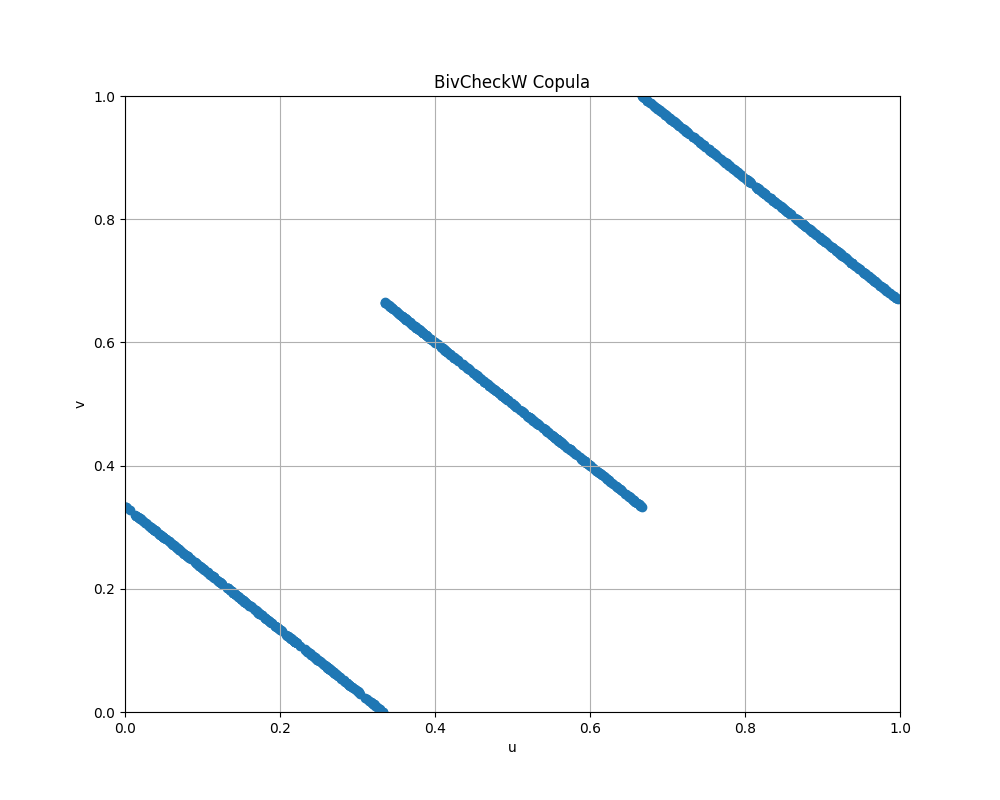} \\
    \includegraphics[width=0.32\textwidth]{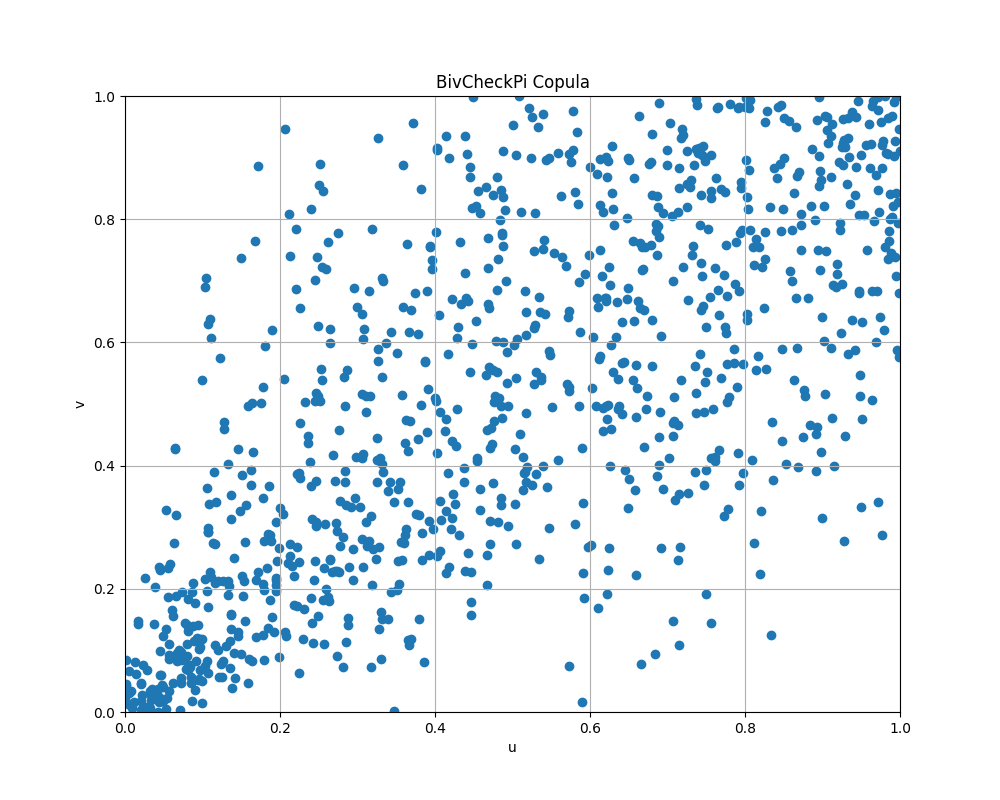}
    \includegraphics[width=0.32\textwidth]{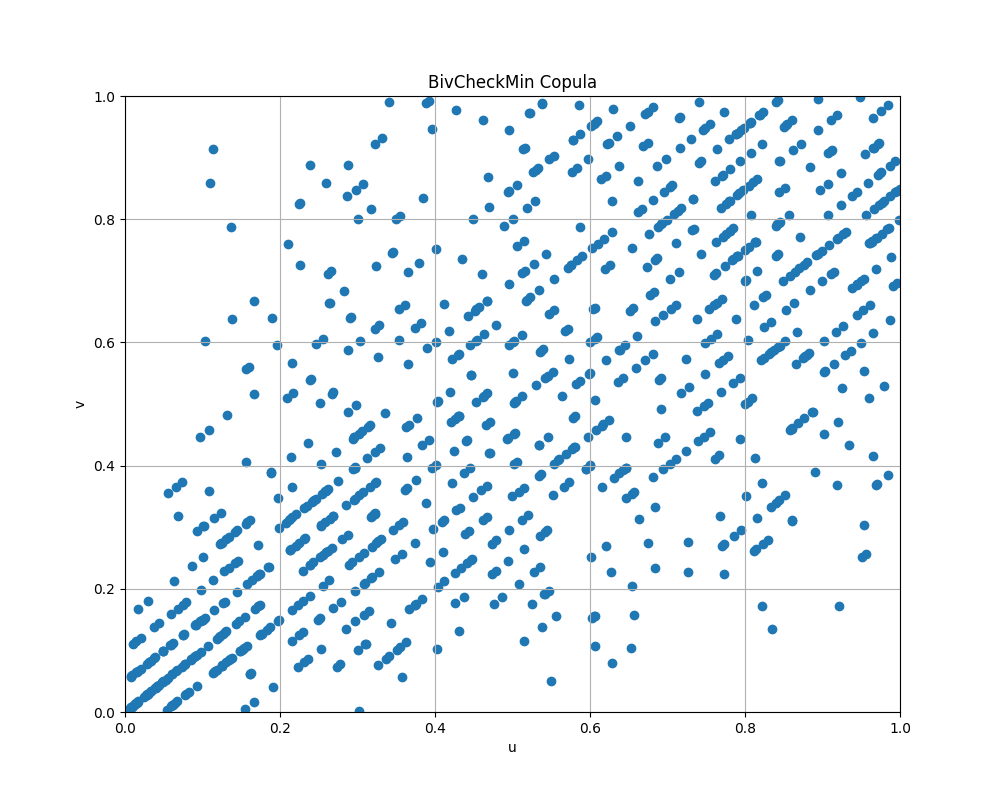}
    \includegraphics[width=0.32\textwidth]{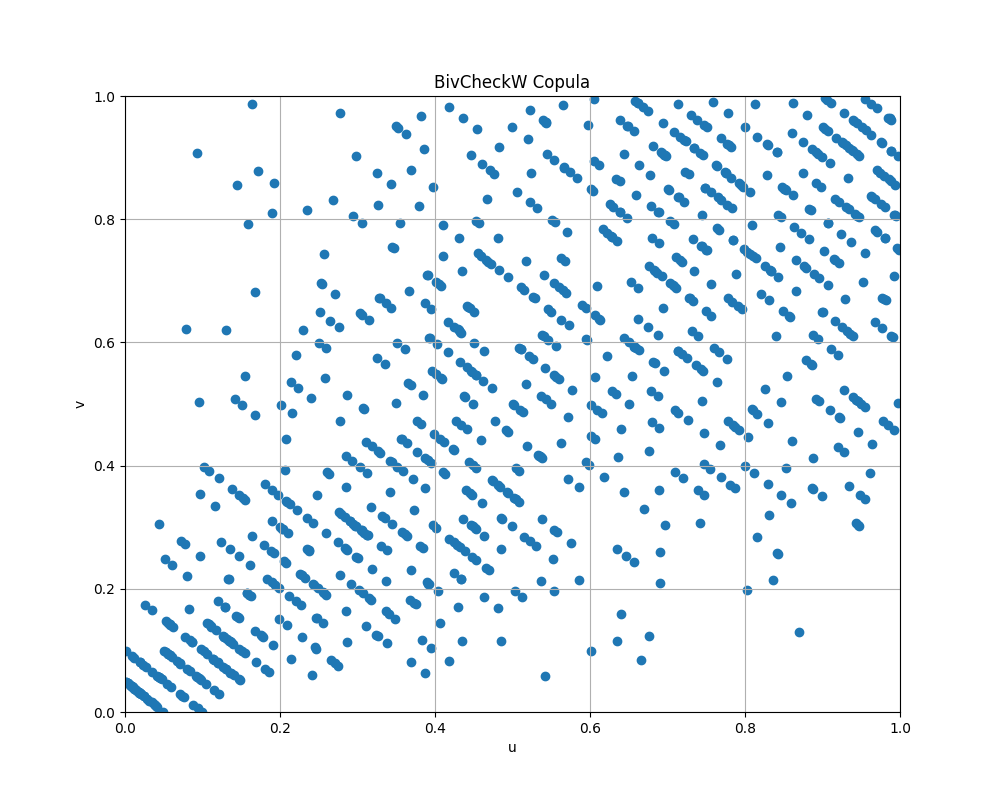}
    \caption{
        Scatter plots of $1,000$ samples generated from checkerboard-type copulas.
        At the top: Checkerboard (left), check--min (middle) and check--w (right) associated with the checkerboard matrix $\frac13 I_3$.
        At the bottom: $20\times 20$-checkerboard approximations of the Clayton copula with parameter $\theta=2$ using the checkerboard copula (left), the check--min copula (middle) and the check--w copula (right).
    }
    \label{fig:checkerboard_types}
\end{figure}

One may also consider a generalization of the check--min and check--w copulas.
For a copula $C$, we say that $C$ is an \emph{$m\times n$-perfect dependence copula} if, for some $(X,Y)\sim C$, the following holds for every $1\le i\le m$ and $1\le j\le n$ with $\Prob[(X,Y)\in\II_{i,j}]>0$:
conditionally on $(X,Y)\in\II_{i,j}$,
\[
    mX-i+1 \sim \Unif(0,1),
\]
and
\begin{align}\label{eq:check_pd_copula}
    Y=\frac{f_{i,j}(mX-i+1) + j - 1}n
\end{align}
almost surely for some Lebesgue measure preserving function $f_{i,j}:[0,1]\rightarrow[0,1]$.
Here, being \emph{Lebesgue measure preserving} means that
\[
    \int_{0}^{1} g(f_{i,j}(x)) \de x = \int_{0}^{1} g(y) \de y
\]
for all bounded, measurable functions $g:[0,1]\rightarrow\mathbb{R}$.
We let $\CC^{\Delta}_{\text{pd}}$ denote the set of all $m\times n$-perfect dependence copulas associated with an $m\times n$-checkerboard matrix $\Delta$.
When choosing $f_{i,j}(x)=x$ or $f_{i,j}(x)=1-x$ for all $1\leq i \le m, 1\le j \le n$, one obtains back the formulas \eqref{eq:check_min_copula_2} and \eqref{eq:check_w_copula_2}, so that check--min and check--w copulas are special cases of perfect dependence copulas.
We broadly refer to copulas that interpolate each cell using a common local dependence structure as \emph{checkerboard-type} copulas.
A visual illustration distinguishing the different checkerboard-type copulas is given in Figure \ref{fig:checkerboard_types}.

\subsection{Measures of association}
\label{subsec:measures_of_association}

Two classical measures of association are Spearman's rho and Kendall's tau, which provide alternatives to the Pearson correlation coefficient that do not depend on the marginal distributions of the random variables.
Both of them can be expressed as an integral over the unit square $[0,1]^2$.
That is, for a bivariate copula $C$, one can write Spearman's rho as
\begin{align}\label{eq:rho_integral}
    \rho_S(C) = 12 \int_{[0,1]^2} C(u,v) \de \lambda^2(u,v) - 3,
\end{align}
and Kendall's tau as
\begin{align}\label{eq:tau_integral}
    \tau(C) = 1 - 4\int_{[0,1]^2} \partial_1 C(u,v) \partial_2 C(u,v) \de \lambda^2(u,v)
,\end{align}
see, e.g., \cite[Definitions 2.4.5 and 2.4.6]{Durante-2016}.
Note that the integrals in \eqref{eq:tau_integral} are always well-defined, since copulas are almost everywhere partially differentiable, see \cite[Thm.~2.2.7]{Nelsen-2006}.
An equivalent (and classical) interpretation of Kendall’s~$\tau$ is in terms
of \emph{concordant} and \emph{discordant} pairs of observations:
if $(U_1,V_1)$ and $(U_2,V_2)$ are two independent draws from the copula~$C$,
then
\begin{align}\label{eq:tau_concordance}
   \tau(C)
   = \Prob\left[(U_1-U_2)(V_1-V_2)>0\right] - \Prob\left[(U_1-U_2)(V_1-V_2)<0\right],
\end{align}
i.e.\ the probability of concordance minus the probability of discordance, see, e.g., \cite[Section 5.1.1]{Nelsen-2006}.
This probabilistic view is particularly handy for copulas supported on
discrete sets such as shuffle--of--min constructions (see
Section~\ref{subsec:shuffle_of_min_measures}). \\
Next to these two measures, which take values in $[-1,1]$, it is also interesting to measure the strength of dependence between two random variables $X$ and $Y$.
Chatterjee's xi is one way to do this, yielding values in $[0,1]$, where the value $0$ is consistent with independence between $X$ and $Y$, and $1$ with perfect dependence, i.e. $Y = f(X)$ for some measurable function $f$, see \cite[Theorem~2.2]{ansarifuchs2022asimpleextension}.
Like Spearman's rho and Kendall's tau, Chatterjee's xi can be expressed as an integral.
For a bivariate copula $C$, it is
\begin{align}\label{eq:xi}
    \xi(C) = 6 \int_{[0,1]^2} \left(\partial_1 C(u,v)\right)^2 \de \lambda^2(u,v) - 2
,\end{align}
compare \cite{dette2013copula} and \cite{chatterjee2020}.
For checkerboard, check--min and check--w copulas, the above integral formulas for Kendall's tau, Spearman's rho and Chatterjee's xi can be evaluated explicitly in terms of the underlying checkerboard matrix. \\
Further classical measures of association for bivariate copulas are the tail dependence coefficients, see, e.g., \cite{Joe-1997,Nelsen-2006}.
For a given bivariate copula $C$, the \emph{lower tail dependence coefficient} is defined by
\begin{align}\label{eq:lower_tail_dependence_coefficient}
    \lambda_L(C)=\lim _{t \rightarrow 0^{+}} \frac{C(t, t)}{t}
,\end{align}
and the \emph{upper tail dependence coefficient} by
\begin{align}\label{eq:upper_tail_dependence_coefficient}
    \lambda_U(C) =2-\lim _{t \rightarrow 1^{-}} \frac{1-C(t, t)}{1-t}
.\end{align}

\section{Explicit measures of association for approximating copulas}
\label{sec:explicit_measures_of_association}

In this section, we formulate the explicit expressions for Spearman's rho, Kendall's tau, Chatterjee's xi and the tail dependence coefficients for $m\times n$-checkerboard matrices associated with Bernstein, checkerboard, check--min and check--w copulas.

\subsection{Explicit measures of association for Bernstein copulas}
\label{subsec:bernstein_copulas_measures_of_association}

Let $\Gamma$ be the $m\times n$-matrix with constant entries
\[
   \Gamma_{i,j} = \frac{1}{(m+1)(n+1)},
   \quad 1\le i\le m, 1\le j\le n.
\]
This matrix will appear in Spearman’s rho for Bernstein copulas.
Let $\Theta^{(m)}$ be the $m\times m$-matrix with entries
\[
   \Theta^{(m)}_{i,j}
   =
   \frac{(i-j)\binom{m}{i}\binom{m}{j}}
        {\left(2m-i-j\right)\binom{2m-1}{i+j-1}},
   \quad
   1\le i,j \le m,
\]
with the convention that $0/0=1$.
Define $\Theta^{(n)}$ analogously (of size $n\times n$).  
These matrices enter into Kendall’s tau.
For Chatterjee’s xi, we introduce two more matrices to handle integrals of Bernstein polynomials and their derivatives.
Let $\Upsilon$ be the $m\times m$-matrix whose $(i,r)$‐entry is
\[
\Upsilon_{i,r} =
\begin{cases}
\displaystyle
\frac{\binom{m}{i}\binom{m}{r}}{(2m-3)\binom{2m-4}{i+r-2}}
\left[ir-\frac{2m(m-1)\binom{i+r}{2}}{(2m-1)(2m-2)}\right],
&\text{if } 1 \le i,r < m,\\[1.5em]
\displaystyle
\frac{m(m-1)(i-m)\binom{m}{i}}{(2m-1)(2m-2)\binom{2m-3}{m+i-2}},
&\text{if }1 \le i < m,r = m,\\[1.5em]
\displaystyle
\frac{m(m-1)(r-m)\binom{m}{r}}{(2m-1)(2m-2)\binom{2m-3}{m+r-2}},
&\text{if }i = m,1 \le r < m,\\[1em]
\frac{m^2}{2m-1},
& \text{if }i = m,r = m.
\end{cases}
\]
and let $\Lambda$ be the $n\times n$-matrix whose $(j,s)$‐entry is
\[
   \Lambda_{j,s}
   :=
   \frac{\binom{n}{j}\binom{n}{s}}{(2n+1)\binom{2n}{j+s}}
.\]
The above matrix definitions give exact formulas for Spearman's rho, Kendall's tau and Chatterjee's xi for arbitrary Bernstein copulas.
These formulas are the content of the following proposition.

\begin{proposition}[Explicit measures of association for Bernstein copulas]
\label{lem:bernstein_measures_of_association}
\ \\
Let $C = C^D_B$ be the Bernstein copula associated with the $m\times n$-grid copula matrix $D$. Then:
\begin{align*}
    \rho_S(C^D_B)
    ~=~& 12\tr\left(\Gamma^\top D\right)-3,
    \\
    \tau(C^D_B)
    ~=~& 1-\tr\left(\Theta^{(m)}D\Theta^{(n)}D^\top\right),
    \\
    \xi(C^D_B)
    ~=~& 6\tr\left(\Upsilon D\Lambda D^\top\right)-2.
\end{align*}
Furthermore, the tail dependence coefficients are given by $\lambda_L(C^D_B) = \lambda_U(C^D_B) = 0$.
\end{proposition}

Here $\tr(\Gamma^\top D)=\sum_{i=1}^m\sum_{j=1}^n \Gamma_{i,j}D_{i,j}$ is the Frobenius inner product of $\Gamma$ and $D$.
In the case of $m=n$, the above formula for Spearman's rho and Kendall's tau can be found in \cite[Theorem 9 and 10]{durrleman2000copulas} and the rectangular case is a direct extension.
The derivation of the formula for Chatterjee's xi in Proposition \ref{lem:bernstein_measures_of_association} is given in the appendix from page \pageref{proof:bernstein_measures_of_association} onwards.

\subsection{Explicit measures of association for shuffle--of--min copulas}
\label{subsec:shuffle_of_min_measures}

Let $C_\pi$ be the order-$n$ straight shuffle--of--min copula determined by a permutation $\pi\in\mathfrak{S}_n$ (equal strip width $1/n$ and no reversals). Denote
\[
    N_{\mathrm{inv}}(\pi)=\#\{(i,j):i<j,\pi(i)>\pi(j)\},
\qquad
    d_i=\pi(i)-i.
\]
The measures of association introduced in Section~\ref{subsec:measures_of_association} admit closed algebraic forms for $C_{\pi}$ that depend only on these permutation statistics.

\begin{proposition}[Explicit measures of association for a straight shuffle--of--min copula]\label{prop:shuffle_measures}~\\
For the equal--width, straight shuffle--of--min copula $C_\pi$ of order $n$, we have
\[
    \tau(C_\pi) = 1-\frac{4N_{\mathrm{inv}}(\pi)}{n^2},
    \quad
    \rho_S(C_\pi) = 1-\frac{6\sum_{i=1}^n d_i^2}{n^3},
    \quad
    \xi(C_\pi) = 1.
\]
Furthermore, $\lambda_L(C_\pi) = 1_{\{\pi(1)=1\}}$ and $\lambda_U(C_\pi) = 1_{\{\pi(n)=n\}}$.
\end{proposition}

A derivation of the formulas in Proposition \ref{prop:shuffle_measures} is given in the appendix from page \pageref{proof:shuffle_measures} onwards.
Note that similar formulas for Spearman's rho and Kendall's tau have been observed in \cite[Lemma 1]{schreyer2017exact} and in the recent \cite[Lemma 3.2]{tschimpke2025revisiting}, with the latter covering the Kendall's tau formula given in the Proposition above in the case of symmetric permutations.
Furthermore, note that the identity for Chatterjee's xi is a direct consequence of the fact that shuffle--of--min copulas are perfect dependence copulas (compare, e.g., \cite[Example 1.1]{ansari2025continuity}), and the tail dependence coefficients are elementary.

\subsection{Explicit measures of association for checkerboard-type copulas}
\label{subsec:checkerboard_measures}
To give concise expressions, we make use of the following matrices:
First, let 
\[
   \Delta =\left(\Delta_{i,j}\right)_{1\le i\le m,1\le j\le n}
\]
be an $m\times n$-checkerboard matrix and denote by $\Delta^\top$ its transpose.
Next, define the $m\times n$-matrix $\Omega$ by
\[
    \Omega_{i,j} := \frac{(2m-2i+1)(2n-2j+1)}{mn}
\]
for $1\le i\leq m,1\le j\leq n$.
Also, let $\Xi^{(m)}$ be the $m\times m$-matrix with entries
\[
   \Xi^{(m)}_{i,j}
   =
   \begin{cases}
     2, & \text{if }i>j,\\
     1, & \text{if }i=j,\\
     0, & \text{if }i<j,
   \end{cases}
   \qquad
   (1\le i,j \le m),
\]
and let $\Xi^{(n)}$ be the analogous $n\times n$-matrix.
Lastly, let $T$ be the strict upper-triangular $n\times n$-matrix
\[
   T_{i,j}
   =\begin{cases}
         1, & \text{if }i<j,\\
         0, & \text{otherwise,}
        \end{cases}
   \qquad (1 \le i,j \leq n)
\]
and let $M_{\xi}$ be the $n\times n$-matrix given by
\[
    M_{\xi} = TT^\top + T^\top + \tfrac{1}{3}I_n
.\]

\begin{proposition}[Explicit measures of association for checkerboard-type copulas]
    \label{lem:checkerboard_measures_of_association}~\\
    Let $C_{\Pi}$, $C_{\nearrow}$ and $C_{\searrow}$ be bivariate checkerboard, check--min and check--w copulas associated with an $m\times n$-checkerboard matrix $\Delta$.
    Then, the measures of association can be expressed as follows:
    \begin{enumerate}[label=(\roman*)]
        \item \label{itm:rho_checkerboard} Spearman's rho:
        \begin{align*}
            \rho_S(C_{\Pi}) &= 3\tr\left(\Omega^\top\Delta\right)-3, \\
            \rho_S(C_{\nearrow}) &= \rho_S(C_{\Pi}) + \frac1{mn}, \\
            \rho_S(C_{\searrow}) &= \rho_S(C_{\Pi}) - \frac1{mn}.
        \end{align*}
        \item \label{itm:tau_checkerboard} Kendall's tau:
        \begin{align*}
            \tau(C_{\Pi}) &= 1
            -
            \tr\left(\Xi^{(m)}\Delta\Xi^{(n)}\Delta^\top\right) \\
            \tau(C_{\nearrow}) &= \tau(C_{\Pi}) + \tr(\Delta^\top \Delta), \\
            \tau(C_{\searrow}) &= \tau(C_{\Pi}) - \tr(\Delta^\top \Delta).
        \end{align*}
        \item \label{itm:xi_checkerboard} Chatterjee's xi:
        \begin{align*}
            \xi(C_{\Pi}) &= \frac{6m}{n}\tr\left(\Delta^\top \Delta M_{\xi}\right) -2, \\
            \xi(C_{\text{pd}}) &= \xi(C_{\Pi}) + \frac{m\tr(\Delta^\top \Delta)}n
        \end{align*}
        for all $C_{\text{pd}}\in\CC^{\Delta}_{\text{pd}}$ and in particular for $C_{\nearrow}$ and $C_{\searrow}$.
    \end{enumerate}
    Furthermore, $\lambda_L(C_{\Pi}) = \lambda_U(C_{\Pi}) = \lambda_L(C_{\searrow}) = \lambda_U(C_{\searrow}) = 0$, $\lambda_L(C_{\nearrow}) = \Delta_{1, 1}(m\wedge n)$ and $\lambda_U(C_{\nearrow}) = \Delta_{m, n}(m\wedge n)$.
\end{proposition}

In the case of $n\times n$-checkerboard copulas the above formula for Spearman's rho and Kendall's tau can be found in \cite[Theorem 15 and 16]{durrleman2000copulas} (see also \cite[Theorem 1 and 2]{sukeda2023minimum} and \cite[Formula (2)]{kuzmenko2020checkerboard}).
We are not aware of references for the other cases.

\begin{corollary}[Upper bound for $\xi(C_{\text{pd}})-\xi(C_{\Pi})$]\label{cor:xi_checkerboard}~\\
    Let $\Delta$ be an $m\times n$-checkerboard matrix and let $C^{\Delta}_{\text{pd}}\in \CC^{\Delta}_{\text{pd}}$.
    Then, it holds that
    \[
        \left|\xi(C^{\Delta}_{\text{pd}}) - \xi(C^{\Delta}_{\Pi})\right|
        \le
        \begin{cases} 
            \frac{m}{n^2}, & \text{if } m\le n \\
            \frac{1}{n}, & \text{if } m > n
        \end{cases}
    .\]
\end{corollary}

The proofs for Proposition \ref{lem:checkerboard_measures_of_association} and Corollary \ref{cor:xi_checkerboard} are given in the appendix from page \pageref{proof:checkerboard_measures_of_association} onwards.

\section{Checkerboard estimates for Chatterjee's xi}
\label{sec:checkerboard_bounds}

In this section, we first discuss in Section \ref{subsec:xi_lower_bound} how the checkerboard and check--min formulas relate to general Chatterjee's xi values, and then analyse their performance as estimates for Chatterjee's xi from sampled data in Section \ref{subsec:xi_estimates}.

\subsection{Checkerboard bound for Chatterjee's xi}
\label{subsec:xi_lower_bound}

The expressions of Proposition \ref{lem:checkerboard_measures_of_association} can be used to calculate the measures of association for a given checkerboard copula in a straightforward and efficient way, without the need for numerical integration or estimates from sampled data.
The next theorem shows that the checkerboard formula in Proposition \ref{lem:checkerboard_measures_of_association} \ref{itm:xi_checkerboard} also serves as lower bound of Chatterjee's xi for every totally positive bivariate copula $C$, and hence shows the formula \eqref{eq:xi_bound}.
A non-negative function \(f\) on \([0,1]^2\) is called TP\(_2\), if
\begin{align}\label{eq:tp2}
  f(u_1,v_1)f(u_2,v_2)
  \le
  f(u_1\wedge u_2,v_1\wedge v_2)
  f(u_1\vee u_2,v_1\vee v_2)
\end{align}
for all \(u_1,u_2,v_1,v_2\in[0,1]\), or for almost all such points if \(f\)
is only defined up to null sets.
Here, $\wedge$ and $\vee$ denote the minimum and maximum, respectively.
Moreover, an absolutely continuous copula \(C\) with density \(c\) is called MTP\(_2\) if its density \(c\) is TP\(_2\), i.e.,
\begin{align}\label{eq:mtp2}
  c\left(u_1,v_1\right)c\left(u_2,v_2\right)
  \le
  c\left(u_1\wedge u_2,v_1\wedge v_2\right)
  c\left(u_1\vee u_2,v_1\vee v_2\right)
\end{align}
for almost all \(u_1,u_2,v_1,v_2\in[0,1]\).

We shall use the following classical inequality due to Karamata, which can be found, e.g., in \cite[Prop.~4.B.1, pp.~156--157]{marshall2011inequalities}.

\begin{lemma}[Karamata's inequality]\label{lem:karamata}
    Let \(a,b\in\mathbb R^m\) be decreasing vectors, i.e.
    \(a_1\ge\cdots\ge a_m\) and \(b_1\ge\cdots\ge b_m\). Suppose that
    \(b\) is majorized by \(a\), that is,
    \[
        \sum_{i=1}^k b_i
        \le
        \sum_{i=1}^k a_i,
        \qquad k=1,\ldots,m-1,
        \qquad
        \sum_{i=1}^m b_i
        =
        \sum_{i=1}^m a_i.
    \]
    Then, for every convex function \(\varphi\colon\mathbb R\to\mathbb R\),
    \(
        \sum_{i=1}^m \varphi(b_i)
        \le
        \sum_{i=1}^m \varphi(a_i).
    \)
\end{lemma}

\begin{theorem}[Checkerboard bound for $\xi$]\label{thm:xi_bounds}
Let \(C\) be an absolutely continuous copula with MTP$_2$ density \(c\).
Let \(\Delta\) be the \(m\times n\)-checkerboard matrix associated with \(C\), and let \(C^\Delta_\Pi\) be the corresponding checkerboard copula.
Then
\[
    \xi(C^\Delta_\Pi)\le \xi(C).
\]
\end{theorem}

\begin{proof}
    If \(m=1\), then the row-sum constraints force \(\Delta_{1,j}=1/n\) for all \(j=1,\ldots,n\), and hence \(C_\Pi^\Delta=\Pi\).
    Thus \(\xi(C_\Pi^\Delta)=0\le \xi(C)\).
    We may therefore assume \(m\ge2\).
    Write
    \(
        h(u,v):=\partial_1 C(u,v)
    \)
    and define for \(i=1,\dots,m\) the row-averaged conditional distribution function
    \begin{align}\label{eq:H_i}
        H_i(v)
        :=
        m\int_{(i-1)/m}^{i/m} h(u,v)\,du
        =
        m\left[
            C\left(\frac{i}{m},v\right)
            -
            C\left(\frac{i-1}{m},v\right)
        \right].
    \end{align}
    By Jensen's inequality, for each \(v\in[0,1]\),
    \[
        \int_0^1 h(u,v)^2\,du
        =
        \sum_{i=1}^m
        \int_{\frac{i-1}m}^{\frac{i}m} h(u,v)^2\,du
        \ge
        \frac1m\sum_{i=1}^m H_i(v)^2.
    \]
    Integrating over \(v\), we obtain
    \begin{align}\label{eq:strip_average_inequality}
        \int_{[0,1]^2} h(u,v)^2\,du\,dv
        \ge
        \frac1m\int_0^1\sum_{i=1}^m H_i(v)^2\,dv.
    \end{align}

    It remains to compare the latter term with the checkerboard copula \(C^\Delta_\Pi\).
    Let \(\mathcal L_n\) denote the piecewise linear interpolation operator on the grid \(\{0,1/n,\ldots,1\}\).
    That is, for a function \(f\colon[0,1]\to\mathbb R\), \(\mathcal L_n f\) is the continuous function satisfying for \(v\in[(j-1)/n,j/n]\),
    \[
        (\mathcal L_n f)(v)
        =
        (j-nv) f\left(\frac{j-1}{n}\right)
        +
        (nv-j+1) f\left(\frac{j}{n}\right),
        \qquad j=1,\ldots,n.
    \]
    By the definition of the checkerboard copula, or equivalently by formula \eqref{frm:check_pi}, one has, for a.e. \(u\in[(i-1)/m,i/m]\),
    \[
        \partial_1 C^\Delta_\Pi(u,v)
        =
        (\mathcal L_n H_i)(v).
    \]
    Indeed, if \(v\in[(j-1)/n,j/n]\), then
    \(
        (\mathcal L_n H_i)(v)
        =
        m\left(
            \sum_{\ell=1}^{j-1}\Delta_{i,\ell}
            +
            \Delta_{i,j}(nv-j+1)
        \right),
    \)
    which is exactly the derivative of \(C^\Delta_\Pi\) with respect to its first coordinate.
    We now show that
    \begin{align}\label{eq:l_n_inequality}
        \frac1m\int_0^1
            \sum_{i=1}^m
            \left((\mathcal L_nH_i)(v)\right)^2\de v
        \le
        \frac1m\int_0^1
            \sum_{i=1}^m H_i(v)^2\de v.
    \end{align}
    For this purpose, set
    \begin{align}\label{eq:a_i}
        a_i(v):=H_i(v)-v,
        \qquad i=1,\dots,m.
    \end{align}
    Then \(\sum_{i=1}^m a_i(v)=0\).
    Since \(c\) is TP$_2$, and total positivity is preserved under integration over intervals, the row-averaged densities
    \[
        H_i'(v)
        =
        m\int_{(i-1)/m}^{i/m} c(u,v)\,du
    \]
    form a nonnegative TP$_2$ function in \((i,v)\).
    Indeed, for \(i<k\) and \(v_1<v_2\),
    \[
    H_i'(v_1)H_k'(v_2)-H_i'(v_2)H_k'(v_1)
    =
    m^2
    \int_{\frac{i-1}m}^{\frac{i}m}
    \int_{\frac{k-1}m}^{\frac{k}m}
    \left[
        c(u,v_1)c(w,v_2)
        -
        c(u,v_2)c(w,v_1)
    \right]\de w\de u
    \ge 0,
    \]
    because \(u<w\) on the domain of integration and \(c\) is TP\(_2\).
    Thus the row-averaged densities form a nonnegative, TP\(_2\) discrete-continuous function in \((i,v)\).
    In particular, the conditional law of \(V\) given \(U\in[(i-1)/m,i/m]\) is stochastically increasing in \(i\).
    Equivalently,
    \(
        H_1(v)\ge H_2(v)\ge \cdots \ge H_m(v)
    \)
    for all \(v\), and therefore by \eqref{eq:a_i}
    \[
        a_1(v)\ge a_2(v)\ge \cdots \ge a_m(v).
    \]
    For \(k=1,\dots,m-1\), define the partial sums
    \(
        A_k(v):=\sum_{i=1}^k a_i(v)
        =
        \sum_{i=1}^k H_i(v)-kv.
    \)
    We claim that each \(A_k\) is concave. Indeed, for a.e. \(v\),
    \[
        A_k'(v)
        =
        m\int_0^{k/m} c(u,v)\,du-k.
    \]
    Since the second marginal is uniform, \(u\mapsto c(u,v)\) is the conditional density of \(U\) given \(V=v\), and therefore
    \[
        A_k'(v)
        =
        m\,\mathbb P\left(U\le \frac{k}{m}\mid V=v\right)-k.
    \]
    By the TP$_2$ property, $C$ is in particular stochastically increasing, see, e.g., \cite[p.~146]{Mueller-Stoyan-2002}.
    Thus
    \[
        v\mapsto \mathbb P\left(U\le \frac{k}{m}\mid V=v\right)
    \]
    is decreasing, and consequently \(A_k'\) is decreasing.
    Hence \(A_k\) is concave.
    Moreover, \(A_k(0)=A_k(1)=0\), so \(A_k\ge0\).

    Now put
    \(
        b_i(v):=(\mathcal L_nH_i)(v)-v
        =
        \mathcal L_n a_i(v),
    \)
    where we use \(\mathcal L_n v=v\).
    Then \(\sum_i b_i(v)=0\), and the partial sums of \(b(v)\) satisfy
    \[
        B_k(v):=\sum_{i=1}^k b_i(v)
        =
        \mathcal L_n A_k(v).
    \]
    Since \(A_k\) is concave, its piecewise linear interpolation lies below it:
    \begin{align}\label{eq:l_n_inequality_2}
        B_k(v)=\mathcal L_n A_k(v)\le A_k(v),
        \qquad k=1,\dots,m-1.
    \end{align}
    Furthermore, the ordering of the components is preserved, since
    \[
        b_i(v)-b_{i+1}(v)
        =
        \mathcal L_n(H_i-H_{i+1})(v)\ge0.
    \]
    Thus, for every fixed \(v\), the vector \(b(v)\) is majorized by \(a(v)\).
    Indeed, both vectors are decreasing and have total sum zero.
    Moreover, for \(k=1,\ldots,m-1\), one has from \eqref{eq:l_n_inequality_2} that
    \[
        \sum_{i=1}^k b_i(v)
        =
        \mathcal L_n A_k(v)
        \le
        A_k(v)
        =
        \sum_{i=1}^k a_i(v).
    \]
    By Lemma~\ref{lem:karamata}, applied to the convex function \(x\mapsto x^2\), we obtain
    \[
        \sum_{i=1}^m b_i(v)^2
        \le
        \sum_{i=1}^m a_i(v)^2.
    \]
    Using \(\sum_i a_i(v)=\sum_i b_i(v)=0\) and \eqref{eq:a_i}, this gives
    \[
        \sum_{i=1}^m \left((\mathcal L_nH_i)(v)\right)^2
        =
        mv^2+\sum_{i=1}^m b_i(v)^2
        \le
        mv^2+\sum_{i=1}^m a_i(v)^2
        =
        \sum_{i=1}^m H_i(v)^2.
    \]
    Therefore, \eqref{eq:l_n_inequality} follows.
    Combining \eqref{eq:l_n_inequality} with \eqref{eq:strip_average_inequality}, we obtain
    \[
    \begin{aligned}
        \int_{[0,1]^2}
            \left(\partial_1 C^\Delta_\Pi(u,v)\right)^2
            \,du\,dv
        &=
        \frac1m\int_0^1
            \sum_{i=1}^m\left((\mathcal L_nH_i)(v)\right)^2\,dv \\
        &\le
        \frac1m\int_0^1
            \sum_{i=1}^mH_i(v)^2\,dv \\
        &\le
        \int_{[0,1]^2}
            \left(\partial_1 C(u,v)\right)^2
            \,du\,dv.
    \end{aligned}
    \]
    Hence, by the integral representation of Chatterjee's \(\xi\) in \eqref{eq:xi},
    \[
        \xi(C^\Delta_\Pi)\le \xi(C),
    \]
    which proves the assertion.
\end{proof}

Figure~\ref{fig:tp2_checkerboard_lower_bound_visualization} illustrates the two smoothing steps in the proof at the density level.

\begin{figure}[htbp]
    \centering
    \includegraphics[width=\linewidth]{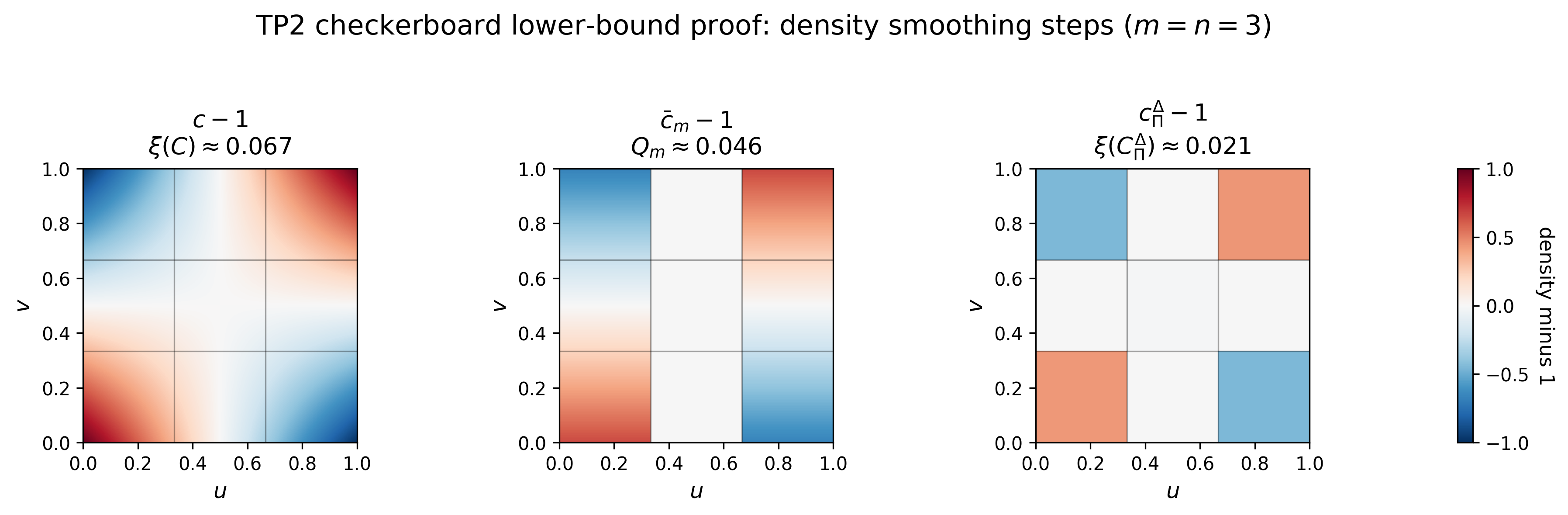}
    \caption{
        Density-level illustration of the two smoothing steps in the proof of
        Theorem~\ref{thm:xi_bounds}, displayed as deviations from independence.
        The left panel shows \(c-1\), the middle panel shows
        \(\bar c_m-1\), obtained by averaging the density over vertical strips,
        and the right panel shows \(c^\Delta_\Pi-1\), obtained by averaging over
        the full checkerboard cells.
        The displayed values correspond to the quantities
        \(6\int(\partial_1 C)^2-2\),
        \(6m^{-1}\int_0^1\sum_{i=1}^m H_i(v)^2\,dv-2\), and
        \(\xi(C^\Delta_\Pi)\), respectively, illustrating the inequalities in
        the proof.
    }
    \label{fig:tp2_checkerboard_lower_bound_visualization}
\end{figure}

The MTP$_2$ assumption in Theorem~\ref{thm:xi_bounds} is essential.
Already when weakening the positive-dependence assumption to stochastic
increasingness, the lower bound no longer holds, as demonstrated by the
following example.

\begin{example}[Lower bound SI counterexample]
\label{ex:xi_checkerboard_not_lower_bound}
Let
\[
    \Delta_{2,4}
    = \frac18
    \begin{pmatrix}
        1 & 2 & 0       & 1 \\
        1 & 0       & 2 & 1
    \end{pmatrix}
\]
and let \(C=C^{\Delta_{2,4}}_{\Pi}\) be the associated \(2\times4\)-checkerboard copula.
The row-conditional distribution functions are ordered, so the associated
checkerboard copula is stochastically increasing.
If one aggregates the four \(y\)-intervals pairwise, the associated coarser \(2\times2\)-checkerboard matrix is
\[
    \Delta_2
    =\frac18
    \begin{pmatrix}
        3 & 1 \\
        1 & 3
    \end{pmatrix}.
\]
Using Proposition~\ref{lem:checkerboard_measures_of_association}
\ref{itm:xi_checkerboard}, we obtain
\[
    \xi(C)
    =
    \xi\left(C^{\Delta_{2,4}}_{\Pi}\right)
    =
    \frac{1}{16},
    \qquad
    \xi\left(C^{\Delta_2}_{\Pi}\right)
    =
    \frac18.
\]
In particular,
\(
    \xi(C)
    <
    \xi\left(C^{\Delta_2}_{\Pi}\right),
\)
so replacing a copula by its coarser checkerboard approximation can increase Chatterjee's \(\xi\).
In particular, the quantity \(\xi(C^\Delta_\Pi)\) is not a lower bound for \(\xi(C)\) in general.
\end{example}

$\xi(C)\leq\xi(C^{\Delta}_{\nearrow})$ is also false in general, as already indicated above.
A simple counterexample is given by the check--min copula $C=C^{\Delta}_{\nearrow}$ associated with the checkerboard matrix
\[
    \Delta = \frac{1}{4}\begin{pmatrix}
        1 & 0 & 0 & 0 \\
        0 & 0 & 1 & 0 \\
        0 & 1 & 0 & 0 \\
        0 & 0 & 0 & 1
    \end{pmatrix}
,\]
which satisfies perfect dependence and hence $\xi(C)=1$, but this is not the case when transitioning to the associated $2\times 2$-checkerboard matrix.
Whilst the lower bound holds under the MTP$_2$ assumption, Example \ref{ex:xi_bounds_counterexample_1} shows that even under these positive-dependence constraints on the copula $C$, check--min copulas may yield lower values for $\xi$ than the copula itself.

\begin{example}[Upper bound MTP$_2$ counterexample]
    \label{ex:xi_bounds_counterexample_1}
    Consider the matrices
    \[
        \Delta_4 \coloneq \frac{1}{4}\begin{pmatrix}
            1 & 0 & 0 & 0 \\
            0 & 0.5 & 0.5 & 0 \\
            0 & 0.5 & 0.5 & 0 \\
            0 & 0 & 0 & 1
        \end{pmatrix},
        \quad \Delta_2 = \frac{1}{8}\begin{pmatrix}
            3 & 1 \\
            1 & 3
        \end{pmatrix}
    ,\]
    and let $C$ be the checkerboard copula associated with $\Delta_4$, i.e. $C = C^{\Delta_4}_{\Pi}$.
    This copula has a totally positive density, which implies multiple other classical dependence concepts, see \cite[Figure 1]{fuchs2023total}.
    Furthermore, using Proposition \ref{lem:checkerboard_measures_of_association} \ref{itm:xi_checkerboard}, it is
    \[
        \xi(C)
        = \xi\left(C^{\Delta_4}_{\Pi}\right)
        = \frac58
        > \frac7{16}
        = \xi\left(C^{\Delta_2}_{\nearrow}\right)
    .\]
\end{example}

\subsection{Checkerboard estimator for Chatterjee's xi}
\label{subsec:xi_estimates}

Let now $(X_1, Y_1), (X_2, Y_2), \ldots$ be a random sample from $(X,Y)$ and assume that $(X,Y)$ has a continuous distribution function.
Chatterjee's xi admits a strongly consistent and asymptotically normal estimator given by
\begin{align}\label{estTn}
    \xi_n(Y|X) = \frac{\sum_{k=1}^n (n \min\{R_k,R_{N(k)}\}-L_k^2)}{\sum_{k=1}^n L_k (n-L_k)},
\end{align}
where $R_k$ denotes the rank of $Y_k$ among $Y_1,\ldots,Y_n$, i.e., the number of $j$ such that $Y_j\leq Y_k$, and $L_k$ denotes the number of $j$ such that $Y_j\geq Y_k$.
For each $k$, the number $N(k)$ denotes the index $j$ such that $X_j$ is the nearest neighbour of $X_k$, where ties are broken uniformly at random.
All these appealing properties allow a fast, model-free variable selection method, noting that $\xi_n$ can be computed in $O(n\log n)$ time.
For more information about the variable selection capabilities of Chatterjee's xi, see \cite{chatterjee2021,huang2022kernel,ansari2024empirical}.
The checkerboard copulas considered above provide an alternative way to estimate Chatterjee's xi.
Let \(0<\kappa\le 1\), set \(K_n=\lfloor n^\kappa\rfloor\), and let \(R_k^X\) and \(R_k^Y\) denote the ranks of \(X_k\) and \(Y_k\), respectively.
For \(r=1,\ldots,n\), put
\(
    J_r^{(n)}=\left[\frac{r-1}{n},\frac{r}{n}\right),
\)
and for \(i=1,\ldots,K_n\), put
\(
    I_i^{(K_n)}=\left[\frac{i-1}{K_n},\frac{i}{K_n}\right).
\)
We define the empirical \(K_n\times K_n\)-checkerboard matrix \(\widehat\Delta_{n,K_n}\) by aggregating the rank-based empirical copula on the fine \(n\times n\) grid onto the coarser \(K_n\times K_n\) grid:
\[
    \left(\widehat\Delta_{n,K_n}\right)_{i,j}
    =
    n\sum_{k=1}^n
    \lambda\!\left(I_i^{(K_n)}\cap J_{R_k^X}^{(n)}\right)
    \lambda\!\left(I_j^{(K_n)}\cap J_{R_k^Y}^{(n)}\right),
    \qquad 1\le i,j\le K_n,
\]
where \(\lambda\) denotes the one-dimensional Lebesgue measure.
This fractional rank-binning construction ensures that \(\widehat\Delta_{n,K_n}\) is a checkerboard matrix, i.e.,
\[
    \sum_{j=1}^{K_n}\left(\widehat\Delta_{n,K_n}\right)_{i,j}
    =
    \frac1{K_n},
    \qquad
    \sum_{i=1}^{K_n}\left(\widehat\Delta_{n,K_n}\right)_{i,j}
    =
    \frac1{K_n}.
\]
For notational brevity below, write
\[
    \Delta_{\lfloor n^\kappa\rfloor}:=\widehat\Delta_{n,K_n}.
\]
Let \(M_{\xi,K_n}=T_{K_n}T_{K_n}^\top+T_{K_n}^\top+\tfrac13 I_{K_n}\), where \(T_{K_n}\) is the strict upper-triangular \(K_n\times K_n\)-matrix.
Then set
\begin{align}\label{eq:xi_estimator}
    \xi^{\kappa}_n(\bsX^{(n)}, \bsY^{(n)})
      &=
      6 \tr\left(
          \widehat\Delta_{n,K_n}^\top
          \widehat\Delta_{n,K_n}
          M_{\xi,K_n}
      \right)
      + \tfrac12 \tr\left(
          \widehat\Delta_{n,K_n}^\top
          \widehat\Delta_{n,K_n}
      \right)
      - 2 .
\end{align}
This is the arithmetic average of the checkerboard and local-perfect-dependence formulas for Chatterjee's xi in Proposition \ref{lem:checkerboard_measures_of_association} \ref{itm:xi_checkerboard}.
Since both averaged quantities are values of Chatterjee's xi for copulas, \(\xi_n^\kappa\) takes values in \([0,1]\).
This is in contrast to the \eqref{estTn} estimator, which in general can also take negative values, cf.~\cite[Rem.~3]{chatterjee2021}.
Choosing a checkerboard matrix of size \(K_n\times K_n\) with \(\kappa<1\) avoids overfitting by ensuring that both \(K_n\to\infty\) and the average number of observations per grid cell grows as \(n\to\infty\).
In \cite{Ansari-Fuchs-2023}, using checkerboard copulas for estimation has already been done for a whole set of dependence measures that in particular covers Chatterjee's xi, though with a more implicit formula for the estimator.

\begin{theorem}[Convergence of $\xi^{\kappa}_n$]\label{thm:xi_estimator}~\\
    If \(0<\kappa\leq 1/3\), then the estimator $\xi^{\kappa}_n$ can be computed in time $\OO(n\log n)$ and converges to $\xi(C)$ almost surely as $n\to\infty$, where $C$ is the copula associated with $(X,Y)$.
\end{theorem}

\begin{proof}
    The ranks \(R_k^X\) and \(R_k^Y\) can be computed in time \(\OO(n\log n)\).
    Since \(K_n\le n\) for all sufficiently large \(n\), each rank interval \(J_r^{(n)}\) intersects at most two intervals of the coarser \(K_n\)-grid in each coordinate.
    Hence the construction of \(\widehat\Delta_{n,K_n}\) requires only \(\OO(n)\) nonzero overlap contributions after the rank transformation.
    Matrix multiplication of \(K_n\times K_n\)-matrices is possible in \(\OO(K_n^3)\) time, and \(K_n^3=\OO(n^{3\kappa})\le \OO(n)\) whenever \(\kappa\le 1/3\).
    Thus the rank computation is the bottleneck, and the overall computational complexity is \(\OO(n\log n)\).

    The almost sure convergence of the checkerboard part of \(\xi_n^\kappa\) to \(\xi(C)\) follows as in \cite[Theorem 4.2]{Ansari-Fuchs-2023}.
    It remains to note that the local-perfect-dependence correction vanishes.
    Since \(\widehat\Delta_{n,K_n}\) is a checkerboard matrix, each of its row sums equals \(1/K_n\), and therefore
    \[
        \tr\!\left(
            \widehat\Delta_{n,K_n}^\top
            \widehat\Delta_{n,K_n}
        \right)
        =
        \sum_{i,j=1}^{K_n}
        \left(\widehat\Delta_{n,K_n}\right)_{i,j}^2
        \le
        \sum_{i=1}^{K_n}
        \left(
            \sum_{j=1}^{K_n}
            \left(\widehat\Delta_{n,K_n}\right)_{i,j}
        \right)^2
        =
        \frac{1}{K_n}
        \xrightarrow[n\to\infty]{}0,
    \]
    because \(K_n=\lfloor n^\kappa\rfloor\to\infty\) for \(0<\kappa\le 1/3\).
    Hence \(\xi_n^\kappa\to \xi(C)\) almost surely.
\end{proof}

The estimator $\xi_n^\kappa$ in \eqref{eq:xi_estimator} is the arithmetic mean of two natural checkerboard quantities: the lower checkerboard value
\[
    \underline{\xi_n^\kappa}
    =
    6 \tr\!\left(
        \Delta_{\lfloor n^\kappa\rfloor}^{\top}
        \Delta_{\lfloor n^\kappa\rfloor}
        M_\xi
    \right)-2
\]
corresponding to the checkerboard copula $C_\Pi^\Delta$, and the upper local-perfect-dependence value
\[
    \overbar{\xi_n^\kappa}
    =
    6 \tr\!\left(
        \Delta_{\lfloor n^\kappa\rfloor}^{\top}
        \Delta_{\lfloor n^\kappa\rfloor}
        M_\xi
    \right)
    +
    \tr\!\left(
        \Delta_{\lfloor n^\kappa\rfloor}^{\top}
        \Delta_{\lfloor n^\kappa\rfloor}
    \right)
    -2
\]
corresponding to $C_\nearrow^\Delta$.
Thus
\[
    \xi_n^\kappa
    =
    \frac12
    \left(
        \underline{\xi_n^\kappa}
        +
        \overbar{\xi_n^\kappa}
    \right).
\]
The lower quantity tends to be conservative because it replaces the within-cell dependence by independence, whereas the upper quantity inserts perfect positive dependence within each occupied cell.
The mean therefore balances two opposite finite-sample effects, while still taking values in $[0,1]$.
This is in contrast to the nearest-neighbour estimator $\xi_n$ in \eqref{estTn}, which can be negative in finite samples; see also \cite[Remark~3]{chatterjee2021}.

We now turn to the finite-sample behaviour of the checkerboard estimators, first examining the effect of the grid resolution and then comparing them with the classical nearest-neighbour estimator $\xi_n$ from \eqref{estTn}.
The tuning parameter is the grid exponent $\kappa$.
A value of $\kappa$ close to zero gives a very coarse checkerboard and hence smooths away dependence, whereas a value close to one produces many sparsely populated cells and can overfit the empirical copula.
The theoretical constraint in Theorem~\ref{thm:xi_estimator} suggests the largest computationally efficient choice, namely $\kappa=1/3$, since the matrix operations then remain of order $\OO(n\log n)$ after the rank transformation.
The following example illustrates the resulting bias--variance trade-off.

\begin{example}[Choice of the grid exponent]\label{ex:xi_estimator}
    Consider the Gaussian single-factor model
    \begin{align}\label{eq:single_factor_model}
        Z \sim \mathcal{N}(0,1), 
        \qquad
        \varepsilon \sim \mathcal{N}(0,1),
        \qquad
        X = Z + \varepsilon,
    \end{align}
    where $Z$ and $\varepsilon$ are independent.
    Then $(Z,X)$ is jointly Gaussian with correlation $1/\sqrt{2}$, and
    \[
        \xi(X|Z)
        =
        \frac{3}{\pi}
        \arcsin\left(\frac34\right)
        -\frac12
        \approx 0.3098,
    \]
    see \cite[Proposition~2.7]{ansarifuchs2022asimpleextension}.
    Figure~\ref{fig:xi_estimator_per_kappa} compares the checkerboard estimators for different values of $\kappa$.
    Small values of $\kappa$ lead to overly coarse approximations, while large values increase the finite-sample variability.
    The choice $\kappa=1/3$ provides a stable compromise and is used throughout the simulation study below. 

    \begin{figure}[htbp]
        \centering
        \includegraphics[width=\textwidth]{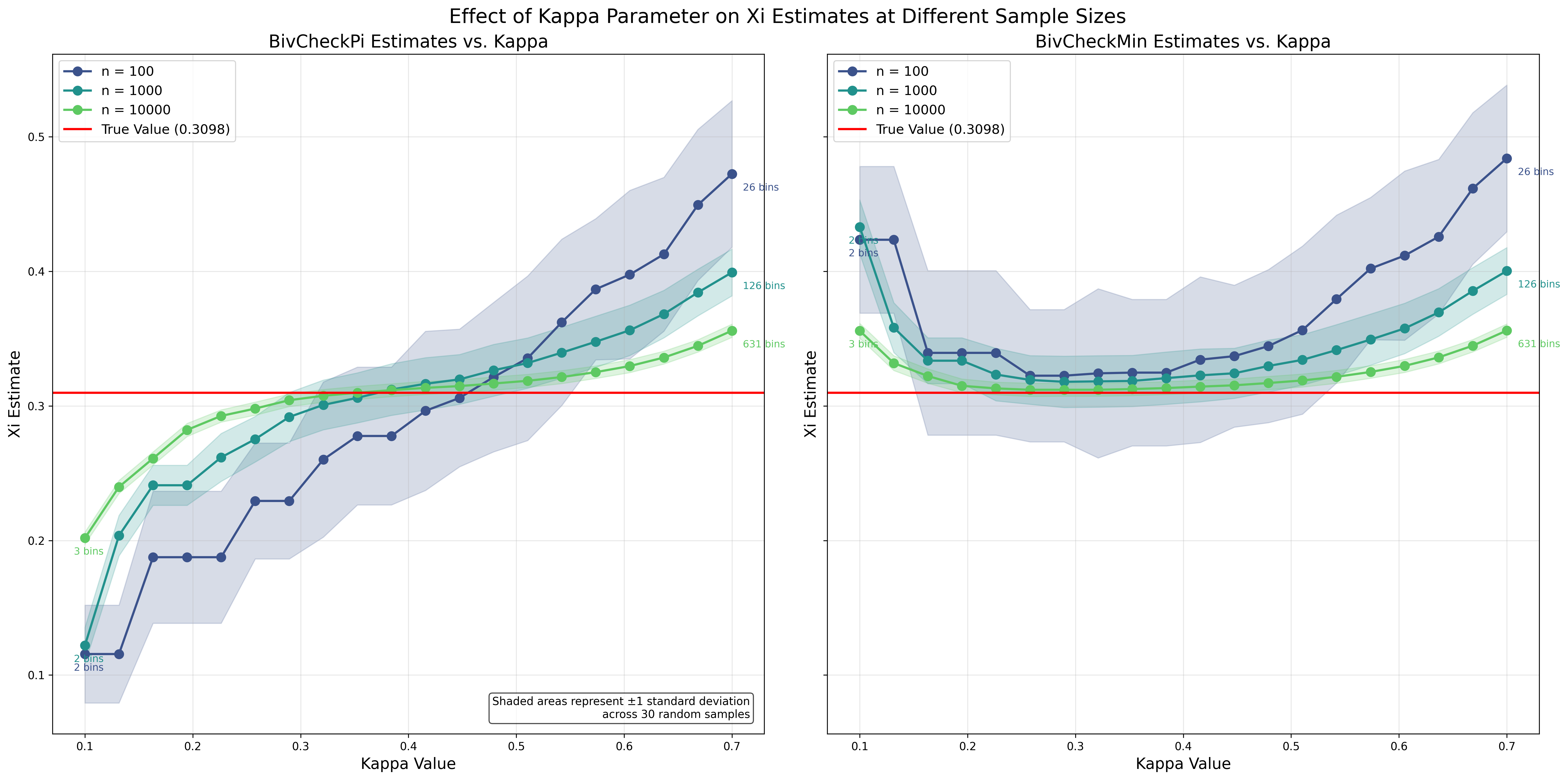}
        \caption{
            Sensitivity of the checkerboard estimator to the grid exponent $\kappa$ in the Gaussian single-factor model \eqref{eq:single_factor_model}.
            The horizontal line marks the theoretical value of $\xi$.
            Left side shows $\underline{\xi_n^{\kappa}}$ (\emph{CheckPi}), right side $\overbar{\xi_n^{\kappa}}$ (\emph{CheckMin}).
        }
        \label{fig:xi_estimator_per_kappa}
    \end{figure}
\end{example}

Having fixed the grid exponent at \(\kappa=1/3\), we next compare the checkerboard estimators with the classical nearest-neighbour estimator \(\xi_n\) from \eqref{estTn}.
In Figure \ref{fig:convergence}, we compare the precision of our implementations of $\xi^{\kappa}_n$, $\overbar{\xi^{\kappa}_n}$ and $\underline{\xi^{\kappa}_n}$ for $\kappa=\frac13$ with standard implementations of the $\xi_n$ estimator using sample data from the model in \eqref{eq:single_factor_model}.
Figure \ref{fig:simulation_speed} shows that also in terms of performance these estimators do not fall behind the standard implementations of $\xi_n$ in the \texttt{xicorpy} and \texttt{scipy} packages.

\begin{figure}[ht!]
    \centering
    \includegraphics[width=0.68\textwidth]{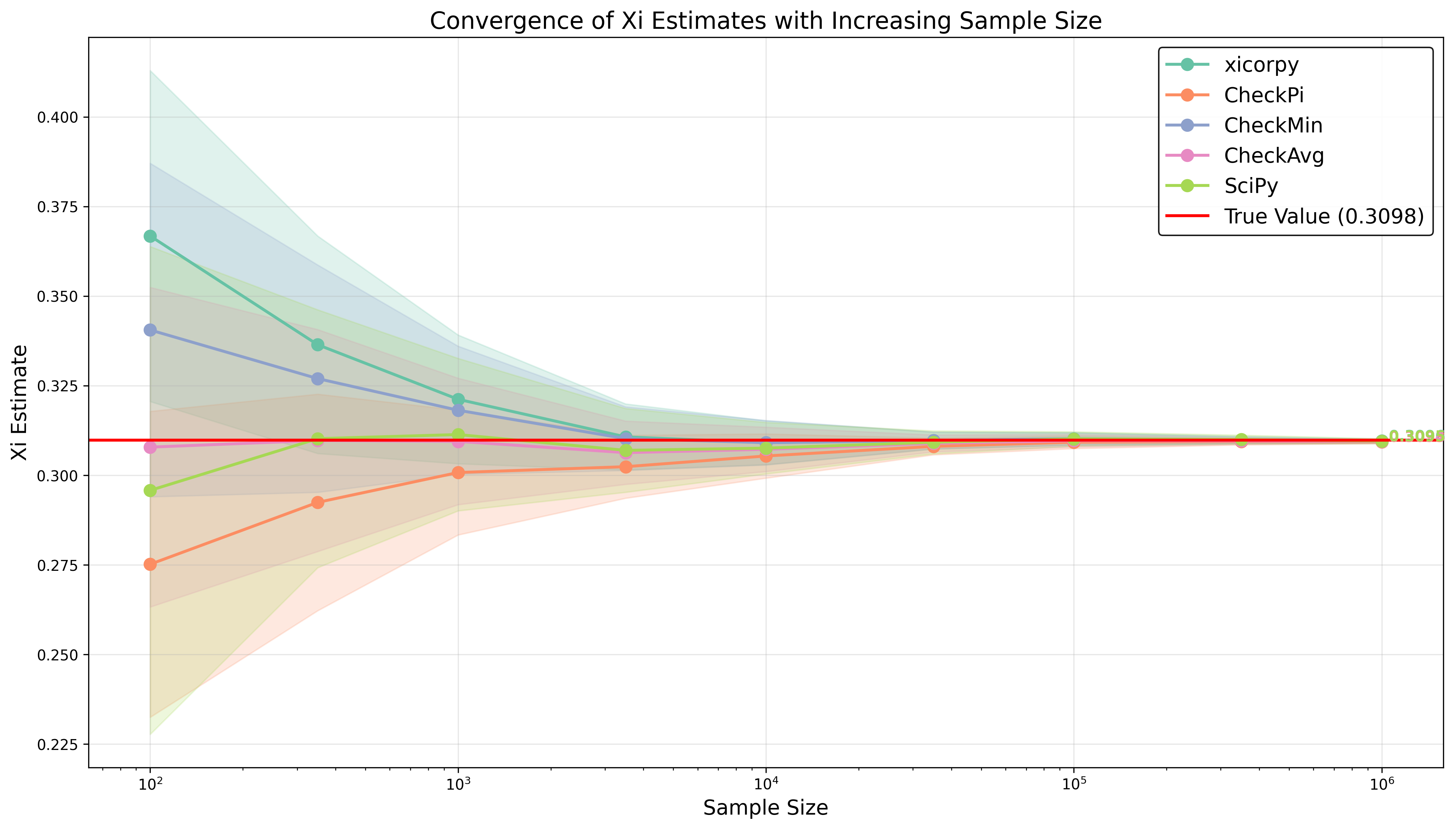}
    \caption{
        Convergence of $\xi$ estimates to the true value as sample size increases.
        The checkerboard estimate $\underline{\xi_n^{\kappa}}$ (\emph{CheckPi}) tends to underestimate the true value, while the check--min estimate $\overbar{\xi_n^{\kappa}}$ (\emph{CheckMin}) tends to overestimate it.
        $\xi^{\kappa}_n$ (\emph{CheckAvg}) is the closest to the true value at smaller sample sizes in this setting.
    }
    \label{fig:convergence}
\end{figure}

\begin{figure}[htbp]
  \centering
  \includegraphics[width=.6\linewidth]{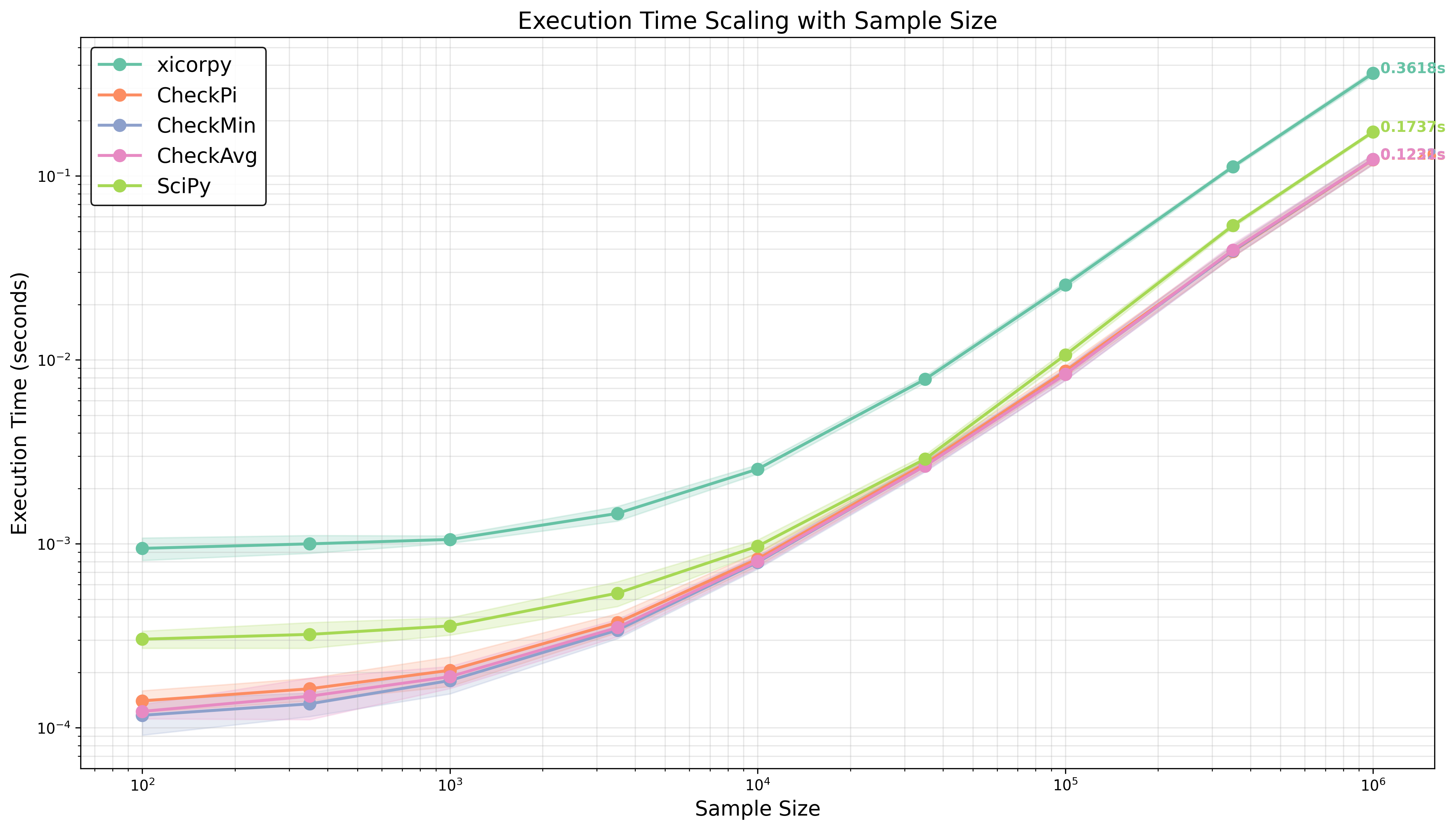}
    \caption{
        Execution time scaling for different estimation methods with increasing sample size.
        Our implementation of $\xi^{\kappa}_n$ outperforms the implementations of $\xi_n$ in \texttt{xicorpy} approximately by a factor of three and the implementation in \texttt{scipy} by approximately 30 \%.
    }
  \label{fig:simulation_speed}
\end{figure}

We now compare $\xi_n^\kappa$ with the nearest-neighbour estimator $\xi_n$ from \eqref{estTn} for different copula families and different levels of dependence.
More precisely, the simulation uses four copula families---Gaussian, Clayton,
Gumbel--Hougaard and Frank---at three calibrated dependence levels.  The labels
low, moderate and strong correspond approximately to the target values
\[
    \xi \approx 0.05,\qquad
    \xi \approx 0.30,\qquad
    \xi \approx 0.65,
\]
respectively.  For each copula family, the parameter is chosen so that the corresponding reference value of Chatterjee's \(\xi\) is close to the target.
The resulting parameter values are therefore comparable across families in the sense that they correspond to approximately the same value of Chatterjee's \(\xi\), although the local dependence structure may still differ substantially between copula families.
For each family, dependence level and sample size
\[
    n\in\{100,200,500,1000,2000,5000\},
\]
we generate $300$ Monte Carlo samples.
For the Gaussian copula the reference value is computed from the closed form
\[
    \xi(C_\rho)
    =
    \frac{3}{\pi}
    \arcsin\!\left(\frac{1+\rho^2}{2}\right)
    -\frac12,
\]
see \cite[Ex.~4]{fuchs2024quantifying}, whereas for the remaining families we use a large-sample approximation based on $3\cdot 10^5$ observations.
The reported root mean squared error (RMSE) values are therefore Monte Carlo
estimates of the finite-sample error relative to a fixed reference value.
The aggregated results are summarized in Table~\ref{tab:simulation_summary},
while Figure~\ref{fig:simulation_rmse} displays the corresponding RMSE curves
separately by copula family and dependence level.
Figure~\ref{fig:simulation_boxplots} complements these summaries by showing the
full distribution of selected estimates at moderate dependence.
The comparison reveals a clear dependence-regime effect.
For weak dependence, the lower checkerboard estimator $\underline{\xi_n^\kappa}$ is particularly competitive.
This is consistent with the construction: when the underlying copula is close to independence, replacing the within-cell structure by independence introduces little structural error and reduces sampling variability.
For strong dependence, the nearest-neighbour estimator $\xi_n$ is typically preferable, since local rank information is well suited to detecting nearly functional relationships.
The intermediate regime is the most balanced one: here the upper checkerboard value $\overbar{\xi_n^\kappa}$ and the averaged estimator $\xi_n^\kappa$ often perform close to, and sometimes better than, the nearest-neighbour estimator.

\begin{table}[htbp]
  \centering
  \begin{tabular}{llcccc}
    \toprule
    Dep. level & $n$ & \textsc{CheckAvg} & \textsc{CheckPi} & \textsc{CheckMin} & $\xi_n$ \\
    \midrule
    Low & 100 & $0.054^{***}$ & $0.031^{***}$ & $0.087$ & $0.068$ \\
     & 500 & $0.023^{***}$ & $0.016^{***}$ & $0.033$ & $0.030$ \\
     & 1000 & $0.016^{***}$ & $0.012^{***}$ & $0.022$ & $0.021$ \\
     & 5000 & $0.006^{***}$ & $0.005^{***}$ & $0.007^{***}$ & $0.009$ \\
    \addlinespace[0.25em]
    \midrule
    Moderate & 100 & $0.055^{***}$ & $0.073$ & $0.073$ & $0.070$ \\
     & 500 & $0.025^{***}$ & $0.032$ & $0.028^{***}$ & $0.032$ \\
     & 1000 & $0.018^{***}$ & $0.022$ & $0.019^{***}$ & $0.022$ \\
     & 5000 & $0.008^{***}$ & $0.009^{***}$ & $0.008^{***}$ & $0.011$ \\
    \addlinespace[0.25em]
    \midrule
    Strong & 100 & $0.091$ & $0.157$ & $0.051$ & $0.046$ \\
     & 500 & $0.040$ & $0.067$ & $0.020$ & $0.019$ \\
     & 1000 & $0.028$ & $0.045$ & $0.014$ & $0.014$ \\
     & 5000 & $0.010$ & $0.015$ & $0.006^{***}$ & $0.006$ \\
    \bottomrule
  \end{tabular}
\caption{
    Family-averaged RMSE of Chatterjee \(\xi\) estimators with \(\kappa=\tfrac{1}{3}\).
    Entries are RMSEs averaged over the Gaussian, Clayton, Gumbel--Hougaard, and Frank copula families.
    Each cell is based on \(300\) Monte Carlo replications per copula family.
    Stars indicate significantly smaller squared error than \(\xi_n\), based on paired one-sided \(t\)-tests pooled over families:
    \(^{*}p<0.1\), \(^{**}p<0.05\), \(^{***}p<0.01\).
    The results indicate that \textsc{CheckPi} is particularly competitive in the low-dependence regime, while \textsc{CheckMin} is the most competitive checkerboard variant in the high-dependence regime and approaches the nearest-neighbour estimator.
}
  \label{tab:simulation_summary}
\end{table}

\begin{figure}[htbp]
  \centering
  \includegraphics[width=0.7\linewidth]{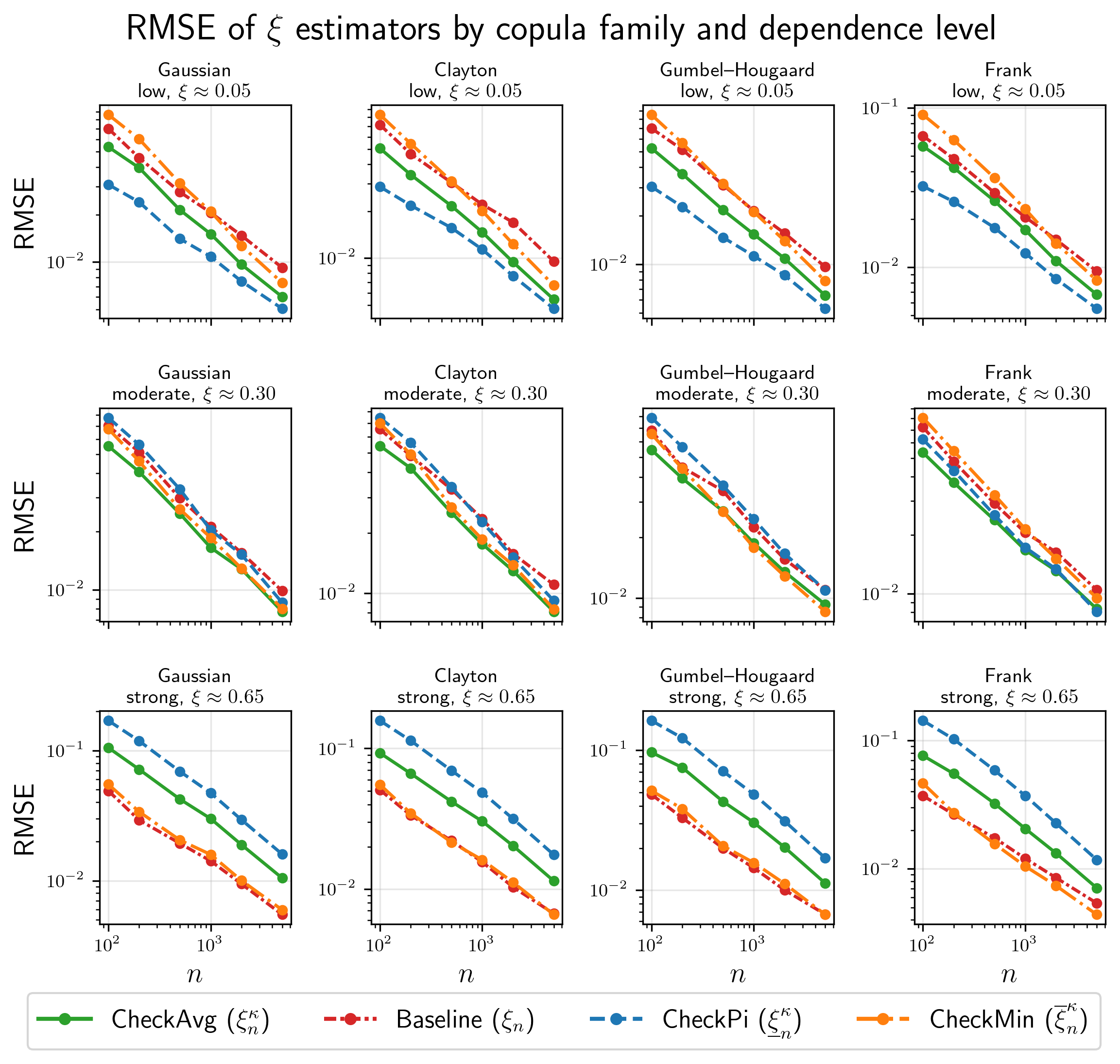}
  \caption{
      RMSE of the four estimators on a log--log scale.
      Columns correspond to copula families and rows to dependence levels.
      Each curve is based on $300$ Monte Carlo replications per setting.
  }
  \label{fig:simulation_rmse}
\end{figure}

\begin{figure}[htbp]
  \centering
  \includegraphics[width=.6\linewidth]{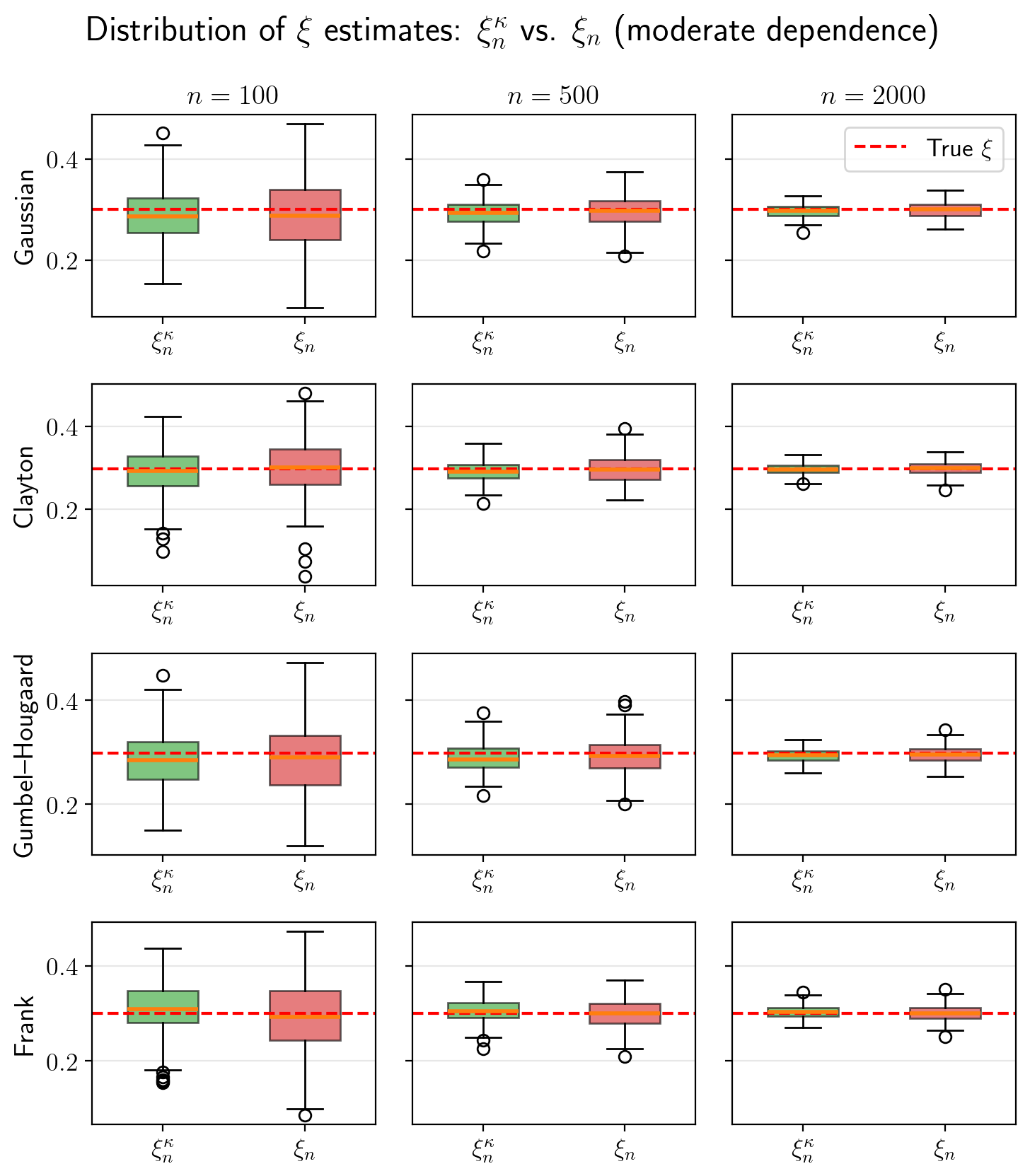}
  \caption{
      Distribution of $\xi_n^\kappa$ and $\xi_n$ across $300$ Monte Carlo replications at moderate dependence.
      The dashed horizontal line marks the reference value of $\xi(C)$.
  }
  \label{fig:simulation_boxplots}
\end{figure}

The averaged estimator $\xi_n^\kappa$ is therefore best viewed as a robust compromise rather than as a uniformly dominating estimator.
It dampens the downward bias of $\underline{\xi_n^\kappa}$ and the upward bias of $\overbar{\xi_n^\kappa}$, remains confined to the natural interval $[0,1]$, and has the same order of computational complexity as $\xi_n$ for the choice $\kappa=1/3$.
This makes it attractive in applications where stable bounded estimates are preferred, especially in low and moderate dependence regimes.

The boxplots in Figure~\ref{fig:simulation_boxplots} illustrate that the
differences reported in Table~\ref{tab:simulation_summary} and
Figure~\ref{fig:simulation_rmse} are not driven only by a few outlying samples.
Although the nearest-neighbour estimator from \eqref{estTn} is strongly
consistent (see \cite{chatterjee2020,chatterjee2021}) and has a well-developed asymptotic theory, the empirical performance of the checkerboard-based estimators remains competitive.
Together with Table~\ref{tab:simulation_summary} and Figure~\ref{fig:simulation_rmse}, this indicates that the explicit checkerboard formulas are not merely useful for analytic copulas: they also yield practical finite-sample estimators of Chatterjee's $\xi$.

\section{Appendix}
\label{sec:appendix}

\begin{proof}[Proof of Proposition~\ref{lem:bernstein_measures_of_association}]
  \label{proof:bernstein_measures_of_association}~\\
  Let $C = C_B^D$ be the Bernstein copula associated with the $m\times n$-grid copula matrix $D$.
  The formulas for Spearman's rho and Kendall's tau can be obtained as straightforward extensions of the computations in \cite[Theorems~9 and~10]{durrleman2000copulas}.
  Upper and lower tail dependence coefficients for Bernstein copulas are always equal to zero due to the boundedness of the density, see, e.g., \cite[Example 1]{pfeifer2016new} for the $m=n$ case.
  The rectangular case is again analogous.
  The rest of the proof is dedicated to deriving the formula for Chatterjee's xi. Recall that \(\xi(C)\) can be written as
  \[
    \xi(C) = 6 \int_{[0,1]^2}\left(\partial_1 C(u,v)\right)^2 \de \lambda^2(u,v)-2
  ,\]
  Hence, we need to evaluate the integral for $C_B^D$.

\medskip\noindent
\textbf{Step 1: Derivative of \(\partial_1 B_{i,m}(u)\).}\\
Write
\begin{align}\label{eq:bernstein_copula_binom}
B_{i,m}(u)=\binom{m}{i}u^i(1-u)^{m-i}.
\end{align}
We distinguish whether $i < m$ or $i=m$.

\smallskip
\noindent
\underline{Case 1:} $1 \le i < m$.\\
A direct product rule and factoring out yields
\[
    \partial_1 B_{i,m}(u)=\binom{m}{i}\left(i - mu\right)u^{i-1}(1-u)^{m-i-1}.
\]
\noindent
\underline{Case 2:} $i = m$.\\
Since $B_{m,m}(u) = u^m$, it is
\[
    \partial_1 B_{m,m}(u)=mu^{m-1}.
\]

\medskip
\noindent
\textbf{Step 2: Derivative of the Bernstein copula.}\\
Recall from \eqref{eq:bernstein_copula} that
\[
    C_B^D(u,v)=\sum_{i=1}^m\sum_{j=1}^n D_{i,j}B_{i,m}(u)B_{j,n}(v).
\]
Thus,
\[
    \partial_1 C_B^D(u,v)=\sum_{i=1}^m\sum_{j=1}^n  D_{i,j}  \partial_1 B_{i,m}(u) B_{j,n}(v).
\]
In the integral for Chatterjee's xi, we need to square this expression and get
\begin{align}\label{eq:partial1_C_B_D}
\left(\partial_1 C_B^D(u,v)\right)^2
= \sum_{i,r=1}^m\sum_{j,s=1}^n  D_{i,j}D_{r,s}  \partial_1 B_{i,m}(u)\partial_1 B_{r,m}(u)  B_{j,n}(v)B_{s,n}(v).
\end{align}

\medskip
\noindent
\textbf{Step 3: Factorize the double integral.}\\
We must integrate \eqref{eq:partial1_C_B_D} over $(u,v)\in[0,1]^2$.
Note that
$\partial_1 B_{i,m}(u)\partial_1 B_{r,m}(u)$ depends \emph{only} on $u$,
while $B_{j,n}(v)B_{s,n}(v)$ depends \emph{only} on $v$.
Hence,
\small
\[
\int_0^1 \int_0^1 \left(\partial_1 C_B^D(u,v)\right)^2\de u\de v
=\sum_{i,r=1}^m\sum_{j,s=1}^n  D_{i,j}D_{r,s}  \Bigl[
    \underbrace{
    \int_0^1 \partial_1 B_{i,m}(u)\partial_1 B_{r,m}(u)\de u }_{=:\Upsilon_{i,r}}
    \Bigr]\Bigl[
    \underbrace{ \int_0^1 B_{j,n}(v)B_{s,n}(v)\de v}_{=:\Lambda_{j,s}}
    \Bigr].
\]
\normalsize
With the matrices $\Upsilon = \left(\Upsilon_{i,r}\right)_{i,r=1}^m$ and $\Lambda = \left(\Lambda_{j,s}\right)_{j,s=1}^n$, we can write the double sum/integral as
\[
\sum_{i,r=1}^m\sum_{j,s=1}^n D_{i,j}D_{r,s}\Upsilon_{i,r}\Lambda_{j,s}
= \tr\left(\Upsilon D\Lambda D^{\mathsf{T}}\right),
\]
so that
\[
\xi\left(C_B^D\right)
=6\tr\left(\Upsilon D\Lambda D^{\mathsf{T}}\right)-2.
\]

\medskip
\noindent
\textbf{Step 4: Explicit form of \(\Lambda\).}\\
By \eqref{eq:bernstein_copula_binom}, we have
\[
    \Lambda_{j,s} = \int_0^1 \binom{n}{j}\binom{n}{s}  v^{j+s}(1-v)^{2n - j-s} \de v.
\]
A standard Beta‐integral identity for nonnegative integers $p,q$ is
\begin{align}\label{eq:beta_integral}
    \int_0^1 x^p (1-x)^q \de x = \frac{p!\,q!}{(p+q+1)!}
,\end{align}
see, e.g., \cite{weisstein2002beta}.
Here, $p = j+s$ and $q = 2n - j-s$, so
\[
    \Lambda_{j,s}
    =\binom{n}{j}\binom{n}{s}\frac{(j+s)!\left(2n - j-s\right)!}{(2n+1)!}
    =\frac{\binom{n}{j}\binom{n}{s}}{(2n+1)\binom{2n}{j+s}},
\]
as specified in Section \ref{subsec:bernstein_copulas_measures_of_association}.

\medskip
\noindent
\textbf{Step 5: Explicit form of \(\Upsilon\).}\\
Recall from Step 1 that in
\[
    \Upsilon_{i,r} = \int_0^1 \partial_1 B_{i,m}(u)\partial_1 B_{r,m}(u)
\]
it is
\[
\partial_1 B_{i,m}(u)
= \begin{cases}
  \binom{m}{i}(i - mu) u^{i-1}(1-u)^{m-i-1}, & (1 \le i < m) \\[0.75em]
  mu^{m-1}, & (i=m)
\end{cases}
.\]
Hence, we must consider four cases for the pair \((i,r)\):

\smallskip
\noindent
\underline{(a)~$1 \le i < m$ and $1 \le r < m$.}\\
Then
\[
\partial_1 B_{i,m}(u)\partial_1 B_{r,m}(u)
= \binom{m}{i}\binom{m}{r}(i - mu)(r - mu)u^{i-1 + r-1}(1-u)^{m-i-1 + m-r-1}.
\]
Expanding $(i - mu)(r - mu)$ yields
\[
\Upsilon_{i,r}
= \binom{m}{i}\binom{m}{r}\int_0^1  \left[ ir - m(i+r)u + m^2u^2  \right]  u^{i+r-2}  (1-u)^{2m- i - r -2}\de u.
\]
Splitting into three Beta‐type integrals:
\begin{multline*}
\Upsilon_{i,r}
= \binom{m}{i}\binom{m}{r} \biggl[
  ir \int_0^1 u^{i+r-2}(1-u)^{2m- i - r -2}\de u \\
  - m(i+r) \int_0^1 u^{i+r-1}(1-u)^{2m- i - r -2}\de u
  + m^2 \int_0^1 u^{i+r}(1-u)^{2m- i - r -2}\de u
\biggr].
\end{multline*}
Using the Beta integrals from \eqref{eq:beta_integral} the three integrals become
\[
\frac{(i+r-2)!\left(2m- i - r -2\right)!}{(2m-3)!},
\quad \frac{(i+r-1)!\left(2m- i - r -2\right)!}{(2m-2)!},
\quad\frac{(i+r)!\left(2m- i - r -2\right)!}{(2m-1)!},
\]
and we now have
\begin{multline*}
\Upsilon_{i,r}
=\binom{m}{i}\binom{m}{r}
\biggl[
ir \frac{(i+r-2)!\left(2m- i - r -2\right)!}{(2m-3)!} \\
- m(i+r)\frac{(i+r-1)!\left(2m- i - r -2\right)!}{(2m-2)!} + m^2\frac{(i+r)!\left(2m- i - r -2\right)!}{(2m-1)!}
\biggr].
\end{multline*}
This expression can be simplified by rewriting each fraction so that all terms share the denominator \((2m-1)!\), i.e. using
\[
    \frac{(i+r-2)!(2m- i - r -2)!}{(2m-3)!}
    =\frac{(i+r-2)!(2m- i - r -2)!\left[(2m-1)(2m-2)\right]}{(2m-1)!},
\]
\[
    \frac{(i+r-1)!(2m- i - r -2)!}{(2m-2)!}
    =\frac{(i+r-1)!(2m- i - r -2)!\left[(2m-1)\right]}{(2m-1)!}
\]
we obtain
\begin{align*}
    \Upsilon_{i,r}
    ~ =~ & \binom{m}{i}\binom{m}{r} \frac{1}{(2m-1)!}    \biggl[ ir(2m-1)(2m-2)(i+r-2)!\left(2m- i - r -2\right)! \\
    & - m(i+r)(2m-1)(i+r-1)!\left(2m- i - r -2\right)! + m^2(i+r)!\left(2m- i - r -2\right)! \biggr]. \\
    =~ &\binom{m}{i}\binom{m}{r}\frac{(i+r-2)!\left(2m- i - r -2\right)!}{(2m-1)! } \left[ (2m-1)(2m-2)ir  -  2m(m-1)\binom{i+r}{2}\right] \\
    =~ & \frac{\binom{m}{i}\binom{m}{r}}{(2m-3)\binom{2m-4}{i+r-2}}
    \left[ ir   -   \frac{2m(m-1)\binom{i+r}{2}}{(2m-1)(2m-2)}\right]
.\end{align*}
\smallskip
\noindent
\underline{(b)~$1 \le i < m$ and $r = m$.}\\
Now, by Step 1 it is
\[
\partial_1 B_{i,m}(u)\partial_1 B_{m,m}(u)
=m\binom{m}{i}  (i - mu)  u^{m + i -2}  (1-u)^{m-i-1}
\]
and
\[
\Upsilon_{i,m}
=m\binom{m}{i}\int_0^1  \left(i - mu\right)  u^{m + i - 2}  (1-u)^{m-i-1}\de u.
\]
Splitting the factor $(i - mu)$:
\[
\Upsilon_{i,m}
= m\binom{m}{i} \biggl[
  i  \int_0^1    u^{m + i - 2}(1-u)^{m-i-1}  \de u 
  -  m  \int_0^1    u^{m + i - 1}(1-u)^{m-i-1}  \de u
\biggr].
\]
Here, use again $p!q!/(p+q+1)!$ with appropriate exponents.
For the first integral, choose $p = (m + i - 2)$ and $q = (m-i-1)$,  and for the second integral choose $p = (m + i - 1)$ and $q = (m-i-1)$.
Then, one gets
\begin{align*}
    \int_0^1  u^{m + i - 2}(1-u)^{m-i-1}\de u
    &=\frac{(m + i - 2)!(m - i - 1)!}{(2m-2)!}, \\
    \quad\int_0^1  u^{m + i - 1}(1-u)^{m-i-1}\de u
    &=\frac{(m + i - 1)!(m - i - 1)!}{(2m-1)!}.
\end{align*}
Thus,
\begin{align*}
    \Upsilon_{i,m}
    \quad = \quad &m\binom{m}{i}\biggl[  i\frac{(m + i - 2)!(m - i - 1)!}{(2m-2)!} -  m\frac{(m + i - 1)!(m - i - 1)!}{(2m-1)!}\biggr] \\
    = \quad & m\binom{m}{i} \frac{(m + i -2)!(m - i -1)!}{(2m-1)!}(m-1)(i-m) \\
    = \quad & \frac{m(m-1)(i-m)\binom{m}{i}}{(2m-1)(2m-2)\binom{2m-3}{m+i-2}}
.\end{align*}

\smallskip
\noindent
\underline{(c)~$i = m$ and $1 \le r < m$.}\\
By symmetry, or by the same direct calculation,
\[
\Upsilon_{m,r}=\frac{m(m-1)(r-m)\binom{m}{r}}{(2m-1)(2m-2)\binom{2m-3}{m+r-2}}
.\]

\smallskip
\noindent
\underline{(d)~$i = m$ and $r = m$.}\\
Here,
\[
\Upsilon_{m,m}
=\int_0^1  \left[mu^{m-1}\right]^2\de u
=m^2\int_0^1  u^{2m-2}\de u
=m^2\left[\frac{u^{2m-1}}{2m-1}\right]_0^1
=\frac{m^2}{2m-1}.
\]
\medskip
Putting these four sub-cases (a)--(d) together provides the complete piecewise formula for \(\Upsilon_{i,r}\) that is specified in Section \ref{subsec:bernstein_copulas_measures_of_association}.
This completes the proof.
\end{proof}

\begin{lemma}[Permutation Sum Identities]\label{lem:perm_identities}~\\
    Let $\pi \in \mathfrak{S}_n$ be a permutation of $\{1, 2, \dots, n\}$ and let $d_i = \pi(i) - i$. Then the following identities hold:
    \begin{enumerate}[label=(\roman*)]
        \item $\displaystyle \sum_{i=1}^{n} d_i = 0$
        \item $\displaystyle \sum_{i=1}^{n} d_i(2i-1) = -\sum_{i=1}^{n} d_i^{2}$
    \end{enumerate}
\end{lemma}
    
\begin{proof}
    We use the properties that for any permutation $\pi \in \mathfrak{S}_n$:
    (a) $\sum_{i=1}^{n} \pi(i) = \sum_{i=1}^{n} i$ and (b) $\sum_{i=1}^{n} \pi(i)^2 = \sum_{i=1}^{n} i^2$.
    The first identity is straightforward.
    For the second identity, let's first expand the left-hand side (LHS):
    \begin{align*}
        \sum_{i=1}^{n} d_i(2i-1)
        \quad&=\quad2\sum_{i=1}^{n} i\pi(i) - \sum_{i=1}^{n} \pi(i) - 2\sum_{i=1}^{n} i^2 + \sum_{i=1}^{n} i \\
        &\stackrelmt{(a)}{=} \quad2\sum_{i=1}^{n} i\pi(i) - \left(\sum_{i=1}^{n} i\right) - 2\sum_{i=1}^{n} i^2 + \left(\sum_{i=1}^{n} i\right) \\
        &= \quad2\sum_{i=1}^{n} i\pi(i) - 2\sum_{i=1}^{n} i^2
    .\end{align*}
    Next, let's expand the term $\sum_{i=1}^{n} d_i^2$ from the right-hand side (RHS):
    \begin{align*}
        \sum_{i=1}^{n} d_i^2
        = \sum_{i=1}^{n} \pi(i)^2 - 2\sum_{i=1}^{n} i\pi(i) + \sum_{i=1}^{n} i^2
        \stackrelmt{(b)}{=} \sum_{i=1}^{n} i^2 - 2\sum_{i=1}^{n} i\pi(i) + \sum_{i=1}^{n} i^2
        = 2\sum_{i=1}^{n} i^2 - 2\sum_{i=1}^{n} i\pi(i)
    .\end{align*}
    Comparing the two resulting terms, we see that:
    \[
        \sum_{i=1}^{n} d_i(2i-1) 
        = - \left( 2\sum_{i=1}^{n} i^2 - 2\sum_{i=1}^{n} i\pi(i) \right) 
        = -\sum_{i=1}^{n} d_i^2.
    \]
    This establishes the second identity.
\end{proof}

\begin{proof}[Proof of Proposition \ref{prop:shuffle_measures}]\label{proof:shuffle_measures}
    First, for Kendall’s tau, let $(U_1,V_1),(U_2,V_2)\sim C_{\pi}$ be independent from each other and write $I=\lceil nU_1\rceil,\,J=\lceil nU_2\rceil$ for the indices of the segments on which the two points fall.
    The random variables $I$ and $J$ are i.i.d.\ and uniform on
    $\{1,\dots ,n\}$, so
    \[
        \mathbb{P}[I=i,J=j]=\frac1{n^2}\qquad(i,j=1,\dots ,n).
    \]
    If $I\neq J$, the sign of $(U_1-U_2)(V_1-V_2)$ is completely determined by the permutation:
    \begin{itemize}
        \item $I<J,\pi(I)<\pi(J)$ or $I>J,\pi(I)>\pi(J)\implies$ concordance;
        \item $I<J,\pi(I)>\pi(J)$ or $I>J,\pi(I)<\pi(J)\implies$ discordance.
    \end{itemize}
    Because $I$ and $J$ are chosen \emph{with} replacement, ties $I=J$ occur with probability $\mathbb{P}[I=J]=1/n$.
    Since $N_{\mathrm{inv}}(\pi)$ denotes the number of inversions of $\pi$, the probabilities are
    \[
        p_{\text{disc}}=\frac{\displaystyle 2N_{\mathrm{inv}}(\pi)}{n^2},
        \quad p_{\text{conc}} =1-p_{\text{disc}}.
    \]
    Hence, by the concordant–discordant definition given in \eqref{eq:tau_concordance}, it is
    \[
        \tau(C_\pi)=p_{\text{conc}}-p_{\text{disc}}
        = 1-\frac{4N_{\mathrm{inv}}(\pi)}{n^2}.
    \]     
    Regarding Spearman’s rho, fix a segment index $i$ and write the rank displacement
    $d_i:=\pi(i)-i$.
    The support of $C_\pi$ is the union of $n$ diagonal line segments
    \[
        S_i\coloneq\left\{\left(\frac{i-1+t}{n},\frac{\pi(i)-1+t}{n}\right)~\middle|~ t\in[0,1]\right\}
    .\]
    Each carries probability mass $1/n$. On $S_i$ the coordinates are related by $V=U+\frac{d_i}{n}$, so $UV=U^2+\frac{d_i}{n}U$.
    With $t\sim\operatorname{Unif}[0,1]$ and $U=(i-1+t)/n$, the conditional expectation is given by
    \[
    \mathbb{E}[U\mid I=i]
    =\int_{0}^{1}\frac{i-1+t}{n}\,dt
    =\frac{2i-1}{2n},
    ,\]
    so that
    \[
    \mathbb{E}[U^{2}\mid I=i]
    =\int_{0}^{1}\frac{(i-1+t)^{2}}{n^{2}}\,dt
    =\frac{(2i-1)^{2}}{4n^{2}}+\frac{1}{12n^{2}}
    \]
    and
    \[
        \mathbb{E}\left[UV\mid I=i\right]
            =\mathbb{E}[U^{2}\mid I=i]+\frac{d_i}{n}\mathbb{E}[U\mid I=i]
            =\frac{(2i-1)^{2}}{4n^{2}}+\frac{1}{12n^{2}}
            +\frac{d_i(2i-1)}{2n^{2}}.
    \]
    Averaging over $i$ one obtains
    \[
        \mathbb{E}[UV]
        = \frac1{n}\sum_{i=1}^{n}\mathbb{E}[UV\mid I=i]
        = \underbrace{\frac1{n}\sum_{i=1}^{n} \left(\frac{(2i-1)^{2}}{4n^{2}} +\frac{1}{12n^{2}}\right)}_{\displaystyle =\mathbb{E}[U^2]=1/3} + \frac1{n}\sum_{i=1}^{n}\frac{d_i(2i-1)}{2n^{2}}.
    \]
    The first sum is $\mathbb{E}[U^2]$ for a $\operatorname{Unif}(0,1)$ variable, which equals $1/3$. 
    For the second sum, we use the second identity from Lemma \ref{lem:perm_identities}, namely $\sum_{i=1}^{n} d_i(2i-1) = -\sum_{i=1}^{n} d_i^{2}$. 
    Substituting yields:
    \[
        \mathbb{E}[UV] 
        = \frac13 + \frac{1}{2n^3} \sum_{i=1}^{n} d_i(2i-1) 
        = \frac13 - \frac{1}{2n^3} \sum_{i=1}^{n} d_i^2.
    \]
    Because $\rho_S(C)=12\mathbb{E}[UV]-3$ for any copula $C$ with uniform margins,
    \[
        \rho_S(C_\pi)
        = 12\left(\frac13-\frac{\sum_{i}d_i^{2}}{2n^{3}}\right)-3
        = 1-\frac{6\sum_{i=1}^{n}d_i^{2}}{n^3}.
    \]
    Next, regarding Chatterjee’s $\xi$, note that for $(X,Y)\sim C_{\pi}$, it is $Y=f(X)$ almost surely, see \cite[Theorem 2.1]{mikusinski1992shuffles}.
    Consequently, using \cite[Theorem 2.1]{ansarifuchs2022asimpleextension}, it follows that $\xi(C_\pi)=1$.
    Lastly, for the tail coefficients, note that for $t<1/n$, $C_\pi(t,t)=t$ if and only if $\pi(1)=1$ (the first segment lies on the main diagonal); otherwise $C_\pi(t,t)=0$.
    Hence, \eqref{eq:lower_tail_dependence_coefficient} yields $\lambda_L=1_{\{\pi(1)=1\}}$.
    A symmetric argument with $t>1-1/n$ gives $\lambda_U=1_{\{\pi(n)=n\}}$, establishing the tail dependence coefficients.
\end{proof}

\begin{proof}[Proof of Proposition~\ref{lem:checkerboard_measures_of_association}]
    \label{proof:checkerboard_measures_of_association}
    ~\begin{enumerate}[label=(\roman*)]
        \item Recall from \eqref{eq:rho_integral} that
        \[
        \rho_S(C) = 12 \int_{[0,1]^2} C(u,v)\de \lambda^2(u,v)-3,
        \]
        where $\lambda^2$ denotes the Lebesgue measure on $[0,1]^2$ and recall also from \eqref{frm:check_pi} that the copula $C_{\Pi}$ is given by
        \small
        \[
            C_{\Pi}(u,v)
            = \sum_{k,l=1}^{i-1,j-1} \Delta_{k,l} + \sum_{k=1}^{i-1}\Delta_{k,j} (nv-j+1) + \sum_{l=1}^{j-1}\Delta_{i,l}(mu-i+1) + \Delta_{i,j} (mu-i+1)(nv-j+1)
        \]
        \normalsize
        for $(u,v) \in \II_{i,j}$.
        Hence, with a simple substitution, it is
        \[
            \int_{\frac{i-1}{m}}^{\frac{i}{m}}\int_{\frac{j-1}{n}}^{\frac{j}{n}} C_{\Pi}(u,v) \de v\de u
            = \frac{1}{mn}\left(\sum_{k,l=1}^{i-1,j-1}\Delta_{k,l} + \frac1{2}\sum_{k=1}^{i-1}\Delta_{k,j} + \frac1{2}\sum_{l=1}^{j-1}\Delta_{i,l} + \frac1{4}\Delta_{i,j}\right)
        .\]
        Considering the full integral, each $\Delta_{i,j}$ appears in the integral of the corresponding cell and all cells with $\Delta_{i'j'}$ with $i'\ge i$ and $j'\ge j$.
        More precisely, it appears in the cells with $i'>i$ and $j'>j$ with weight $1$, in the cells with $i'>i$, $j'=j$ or $i'=i$, $j'>j$ with weight $\frac12$, and in the cell $(i,j)$ itself with weight $\frac14$.
        Consequently,
        \[
        \int_{[0,1]^2} C_{\Pi}(u,v)\de u\de v
        =\sum_{i=1}^{m}\sum_{j=1}^{n} \frac{(2m-2i+1)(2n-2j+1)}{4mn}\Delta_{i,j}
        ,\]
        and thus
        \[
        \rho_S(C_{\Pi})
        =12\sum_{i=1}^{m}\sum_{j=1}^{n}\frac{(2m-2i+1)(2n-2j+1)}{4mn}\Delta_{i,j}-3
        =3 \tr\left(\Omega^\top\Delta\right) - 3
        .\]
        For the check--min copula $C_{\nearrow}$, recall its formula from \eqref{frm:check_min}.
        It follows that
        \[
            \int_{\frac{i-1}{m}}^{\frac{i}{m}}\int_{\frac{j-1}{n}}^{\frac{j}{n}} C_{\nearrow}(u,v) \de v\de u
            = \frac{1}{mn}\left(\sum_{k,l=1}^{i-1,j-1}\Delta_{k,l} + \frac1{2}\sum_{k=1}^{i-1}\Delta_{k,j} + \frac1{2}\sum_{l=1}^{j-1}\Delta_{i,l} + \frac1{3}\Delta_{i,j}\right)
        ,\]
        and therefore
        \[
            \int_{[0,1]^2} C_{\nearrow}(u,v)\de u\de v
            =\sum_{i,j=1}^{m,n} \frac{(2m-2i+1)(2n-2j+1)}{4mn}\Delta_{i,j} + \frac1{12mn}\sum_{i,j=1}^{m,n} \Delta_{i,j}
        .\]
        Hence,
        \[
        \rho_S(C_{\nearrow})
        = \rho_S(C_{\Pi}) + \frac1{mn}
        \]
        as stated.
        Similarly, for $C_{\searrow}$, one obtains from \eqref{frm:check_w} that
        \[
            \int_{\frac{i-1}{m}}^{\frac{i}{m}}\int_{\frac{j-1}{n}}^{\frac{j}{n}} C_{\searrow}(u,v) \de v\de u
            = \frac{1}{mn}\left(\sum_{k,l=1}^{i-1,j-1}\Delta_{k,l} + \frac1{2}\sum_{k=1}^{i-1}\Delta_{k,j} + \frac1{2}\sum_{l=1}^{j-1}\Delta_{i,l} + \frac1{6}\Delta_{i,j}\right)
        ,\]
        leading to the stated result.
        \item Recall from \eqref{eq:tau_integral} that Kendall's tau for $C_{\Pi}$ is given by
        \[
            \tau(C_{\Pi}) = 1 - 4\int_{[0,1]^2} \partial_1 C_{\Pi}(u,v) \partial_2 C_{\Pi}(u,v) \de u \de v
        .\]
        By \eqref{frm:check_pi}, on each cell $\II_{i,j}$, the partial derivative $\partial_1 C_{\Pi}(u,v)$ is given by
        \[
        \partial_1 C_{\Pi}(u,v)
        = m\left(\sum_{l=1}^{j-1} \Delta_{i,l} + \Delta_{i,j} \frac{v-\frac{j-1}n}{\frac{j}n - \frac{j-1}n}\right)
        = m\left(\sum_{l=1}^{j-1} \Delta_{i,l} + \Delta_{i,j} (nv-j+1)\right)
        \]
        for $(u,v)\in\II_{i,j}$.
        An analogous expression holds for $\partial_2 C_{\Pi}(u,v)$ on each cell.
        Integrating cell-by-cell, one obtains
        \begin{align*}
            &\int_0^1\int_0^1 \partial_1 C_{\Pi}(u,v) \partial_2 C_{\Pi}(u,v) \de u \de v \\
            = \quad & \sum_{i,j=1}^{m,n} \int_{\frac{i-1}m}^{\frac{i}m} \int_{\frac{j-1}n}^{\frac{j}n} \partial_1 C_{\Pi}(u,v) \partial_2 C_{\Pi}(u,v) \de u \de v \\
            = \quad & \sum_{i,j=1}^{m,n} \left(m\int_{\frac{j-1}n}^{\frac{j}n} \sum_{l=1}^{j-1} \Delta_{i,l} + \Delta_{i,j} (nv-j+1) \de v \right) \left(n\int_{\frac{i-1}m}^{\frac{i}m} \sum_{k=1}^{i-1} \Delta_{k,j} + \Delta_{i,j} (mu-i+1) \de u \right)\allowdisplaybreaks\\
            = \quad & \sum_{i,j=1}^{m,n}\left(\frac{m}{n}\int_{0}^{1} \sum_{l=1}^{j-1} \Delta_{i,l} + \Delta_{i,j} v \de v \right) \left(\frac{n}{m}\int_{0}^{1} \sum_{k=1}^{i-1} \Delta_{k,j} + \Delta_{i,j} u \de u\right) \allowdisplaybreaks \\ 
            = \quad & \sum_{i,j=1}^{m,n} \left(\sum_{k=1}^{i-1} \Delta_{k,j} + \frac12 \Delta_{i,j}\right)\left(\sum_{l=1}^{j-1} \Delta_{i,l} + \frac12 \Delta_{i,j}\right) \allowdisplaybreaks\\
            = \quad & \frac14 \sum_{i,j=1}^{m,n} \left(\Xi^{(m)} \Delta\right)_{i,j}\left(\Xi^{(n)} \Delta^{\top}\right)_{ji} \\
            = \quad & \frac14 \sum_{i=1}^{m} \left(\Xi^{(m)} \Delta \Xi^{(n)} \Delta^\top\right)_{ii} \\
            = \quad & \frac14 \tr\left(\Xi^{(m)} \Delta \Xi^{(n)} \Delta^\top\right)
        .\end{align*}
        Hence,
        \[
            \tau(C_{\Pi})
            = 1 - 4 \int_{[0,1]^2} \partial_1 C_{\Pi}(u,v) \partial_2 C_{\Pi}(u,v) \de u \de v
            = 1 - \tr\left(\Xi^{(m)} \Delta \Xi^{(n)} \Delta^\top\right)
        \]
        In the case of $C_{\nearrow}$, note that now it is
        \[
            \partial_1 C_{\nearrow}(u,v) 
            =m\left(\sum_{l=1}^{j-1}\Delta_{i,l} + \Delta_{i,j}\1_{\{nv-j+1\geq mu-i + 1\}}\right)
        \]
        for $(u,v)\in\II_{i,j}$ and similarly for $\partial_2 C_{\nearrow}(u,v)$, so that
        \begin{align*}
            &\int_0^1\int_0^1 \partial_1 C_{\nearrow}(u,v) \partial_2 C_{\nearrow}(u,v) \de u \de v \\
            = \quad & \sum_{i,j=1}^{m,n} \int_{\frac{i-1}m}^{\frac{i}m} \int_{\frac{j-1}n}^{\frac{j}n} \partial_1 C_{\nearrow}(u,v) \partial_2 C_{\nearrow}(u,v) \de u \de v \\
            = \quad & \sum_{i,j=1}^{m,n} \int_{0}^{1}\int_{0}^{1} \left( \sum_{l=1}^{j-1} \Delta_{i,l} + \Delta_{i,j} \1_{\{v\geq u\}} \right) \left(\sum_{k=1}^{i-1} \Delta_{k,j} + \Delta_{i,j} \1_{\{u\geq v\}}\right) \de v \de u \allowdisplaybreaks\\
            = \quad & \sum_{i,j=1}^{m,n} \left(\sum_{k=1}^{i-1} \Delta_{k,j}\sum_{l=1}^{j-1} \Delta_{i,l} + \frac12\sum_{k=1}^{i-1} \Delta_{k,j}\Delta_{i,j} + \frac12\sum_{l=1}^{j-1} \Delta_{i,l} \Delta_{i,j}\right) \allowdisplaybreaks\\
            = \quad & \sum_{i,j=1}^{m,n} \left(\sum_{k=1}^{i-1} \Delta_{k,j} + \frac12 \Delta_{i,j}\right)\left(\sum_{l=1}^{j-1} \Delta_{i,l} + \frac12 \Delta_{i,j}\right) - \frac14 \sum_{i,j=1}^{m,n} \Delta_{i,j}^2 \\
            = \quad & \frac14\left( \tr(\Xi^{(m)} \Delta \Xi^{(n)} \Delta^\top) - \tr(\Delta^\top \Delta)\right)
        .\end{align*}
        Consequently,
        \[
            \tau(C_{\nearrow}) = 1 - 4\int_{[0,1]^2} \partial_1 C_{\nearrow}(u,v)\partial_2 C_{\nearrow}(u,v) \de u\de v
            = \tau(C_{\Pi}) + \tr(\Delta^\top \Delta)
        ,\]
        and a similar argument yields
        \begin{align*}
            &\int_0^1\int_0^1 \partial_1 C_{\searrow}(u,v) \partial_2 C_{\searrow}(u,v) \de u \de v \\
            = \quad & \sum_{i,j=1}^{m,n} \int_{0}^{1}\int_{0}^{1} \left( \sum_{l=1}^{j-1} \Delta_{i,l} + \Delta_{i,j} \1_{\{v\geq 1-u\}} \right) \left(\sum_{k=1}^{i-1} \Delta_{k,j} + \Delta_{i,j} \1_{\{u\geq 1-v\}}\right) \de v \de u \allowdisplaybreaks\\
            = \quad & \sum_{i,j=1}^{m,n} \left(\sum_{k=1}^{i-1} \Delta_{k,j}\sum_{l=1}^{j-1} \Delta_{i,l} + \frac12\sum_{k=1}^{i-1} \Delta_{k,j}\Delta_{i,j} + \frac12\sum_{l=1}^{j-1} \Delta_{i,l} \Delta_{i,j} + \int_0^1\int_0^1 \Delta^2_{i,j}\1_{\{u+v\geq 1\}} \de v \de u \right)\allowdisplaybreaks\\
            = \quad & \sum_{i,j=1}^{m,n} \left(\sum_{k=1}^{i-1} \Delta_{k,j} + \frac12 \Delta_{i,j}\right)\left(\sum_{l=1}^{j-1} \Delta_{i,l} + \frac12 \Delta_{i,j}\right) + \frac14 \sum_{i,j=1}^{m,n} \Delta_{i,j}^2 \\
            = \quad & \frac14\left( \tr\left(\Xi^{(m)} \Delta \Xi^{(n)} \Delta^\top\right) + \tr\left(\Delta^\top \Delta\right)\right)
        ,\end{align*}
        which shows
        \[
            \tau(C_{\searrow}) = 1 - 4\int_{[0,1]^2} \partial_1 C_{\searrow}(u,v)\partial_2 C_{\searrow}(u,v) \de u\de v
            = \tau(C_{\Pi}) - \tr(\Delta^\top \Delta)
        .\]
        \item Recall from \eqref{eq:xi} that Chatterjee's xi for a copula $C$ can be expressed as
        \[
            \xi(C)
            =6 \int_0^1\int_0^1 \left(\partial_1 C(u,v)\right)^2 \de u\de v
            -2
        .\]
        For $(u,v)\in\II_{i,j}$, we have
        \begin{align}\label{eq:partial1_xi}
            \partial_1 C_{\Pi}(u,v) 
            = m\left(\sum_{l=1}^{j-1}\Delta_{i,l} + \Delta_{i,j}(nv-j+1)\right)
        .\end{align}
        Hence, squaring and integrating in $v$, one finds
        \begin{align*}
            \int_{\tfrac{j-1}{n}}^{\tfrac{j}{n}} \left(\partial_1 C_{\Pi}(u,v)\right)^2 \de v
            \quad = \quad & \frac{m^2}{n}\int_0^1 \left(\sum_{l=1}^{j-1}\Delta_{i,l} + \Delta_{i,j} v\right)^2 \de v \\
            = \quad & \frac{m^2}{n}\left(\left(\sum_{l=1}^{j-1}\Delta_{i,l}\right)^2 +\left(\sum_{l=1}^{j-1}\Delta_{i,l}\right)\Delta_{i,j} +\tfrac13\Delta_{i,j}^2 \right) \\
            = \quad & \frac{m^2}{n}\left( (\Delta T)_{i,j}^2
            +\left(\Delta T\right)_{i,j}\Delta_{i,j}
            +\tfrac13\Delta_{i,j}^2 \right).
        \end{align*}
        Summing over the cells yields the formula for $\xi(C_{\Pi})$.
        \begin{align*}
            \xi(C_{\Pi}) 
            \quad = \quad & 6 \sum_{i,j=1}^{m,n} \int_{\tfrac{i-1}{m}}^{\tfrac{i}{m}} \int_{\tfrac{j-1}{n}}^{\tfrac{j}{n}} \left(\partial_1 C_{\Pi}(u,v)\right)^2 \de v \de u - 2 \\
            = \quad & \frac{6m}{n} \sum_{i,j=1}^{m,n}\left(\left(\sum_{l=1}^{j-1}\Delta_{i,l}\right)^2 +\left(\sum_{l=1}^{j-1}\Delta_{i,l}\right)\Delta_{i,j} +\tfrac13\Delta_{i,j}^2 \right) -2 \\
            = \quad & \frac{6m}{n} \sum_{i,j=1}^{m,n} \left(\left(\Delta T\right)_{i,j}^2 + \left(\Delta T\right)_{i,j}\Delta_{i,j} +\tfrac13\Delta_{i,j}^2\right) - 2 \\
            = \quad & \frac{6m}{n} \tr\left(\Delta^\top \Delta \left( TT^\top + T^\top + \tfrac{1}{3}I_n\right)\right) - 2 \\
            = \quad & \frac{6m}{n} \tr\left(\Delta^\top \Delta M_{\xi}\right) - 2
        \end{align*}
        In the case of $C_{\text{pd}}\in\CC^{\Delta}_{\text{pd}}$, note that with \eqref{eq:check_pd_copula} it now is
        \[
            \partial_1 C_{\text{pd}}(u,v) 
            =m\left(\sum_{l=1}^{j-1}\Delta_{i,l} + \Delta_{i,j}\1_{\{nv-j+1\geq f_{i,j}(mu-i+1)\}}\right)
        \]
        for $(u,v)\in\II_{i,j}$.
        Note further that due to $f_{i,j}$ being Lebesgue measure preserving, it is
        \[
            \iint_{[0,1]^2} \mathbf{1}\{v \ge f_{i,j}(u)\}\de v \de u
            = \int_0^1 1 - f_{i,j}(u)\de u
            = \frac12
        ,\]
        and hence
        \begin{align*}
            \int_{\tfrac{i-1}{m}}^{\tfrac{i}{m}}\int_{\tfrac{j-1}{n}}^{\tfrac{j}{n}} \left(\partial_1 C_{\text{pd}}(u,v)\right)^2 \de v \de u
            \quad = \quad & \frac{m}n \int_0^1\int_0^1 \left(\sum_{l=1}^{j-1}\Delta_{i,l} + \Delta_{i,j}\1_{\{v\geq f_{i,j}(u)\}}\right)^2 \de v \de u \\	
            = \quad & \frac{m}{n}\left(\left(\sum_{l=1}^{j-1}\Delta_{i,l}\right)^2
            +\left(\sum_{l=1}^{j-1}\Delta_{i,l}\right)\Delta_{i,j}
            +\frac12\Delta_{i,j}^2 \right)
        .\end{align*}
        Thus, one gets an extra $\frac16 \Delta_{i,j}^2$ compared to the previous case, and concludes that
        \[
            \xi(C_{\text{pd}}) = 6 \int_{[0,1]^2} \left(\partial_1 C_{\text{pd}}(u,v)\right)^2 \de u\de v - 2
            = \xi(C_{\Pi}) + \frac{m}n\sum_{i,j=1}^{m,n} \Delta_{i,j}^2
            = \xi(C_{\Pi}) + \frac{m\tr(\Delta^\top \Delta)}{n}
        .\]
        Since $C_{\nearrow}, C_{\searrow}\in \CC^{\Delta}_{\text{pd}}$, this result in particular also holds for them.
    \end{enumerate}    
    Lastly, regarding tail dependence coefficients, it is a direct and classical observation that a copula with a bounded density has no tail dependence, compare, e.g., \cite[below Remark 5.1]{mainik2015risk}.
    In particular, $\lambda_L(C_{\Pi}) = \lambda_U(C_{\Pi}) = 0$.
    For $C_{\searrow}$, since $C_{\searrow} \leq C_{\Pi}$ pointwise, it is clear from the definition of $\lambda_L$ and $\lambda_U$ in \eqref{eq:lower_tail_dependence_coefficient} and \eqref{eq:upper_tail_dependence_coefficient} that 
    \[
        0 \leq \lambda_L(C_{\searrow}) \leq \lambda_L(C_{\Pi}),
        \quad 0\leq \lambda_U(C_{\searrow}) \leq \lambda_U(C_{\Pi})
    .\]
    Hence, also $\lambda_L(C_{\searrow}) = \lambda_U(C_{\searrow}) = 0$.
    Finally, for $C_{\nearrow}$, recall its form from \eqref{frm:check_min}.
    For $t>0$ sufficiently small, it is $(t,t)\in\II_{1,1}$ and thus 
    \[
        \lambda_L(C_{\nearrow}) 
        = \lim_{t\searrow 0} \frac{C_{\nearrow}(t,t)}{t}
        = \lim_{t\searrow 0} \frac{\Delta_{1,1}\min\{nt, mt\}}{t}
        = \Delta_{1,1}(m\wedge n)
    .\]
    Similarly, for $1-t>0$ sufficiently small, it is $(t,t)\in\II_{m,n}$ and thus
    \begin{align*}
        & C_{\nearrow}(1,1) - C_{\nearrow}(t,t) \\
        = \quad & n(1-t)\sum_{k=1}^{m-1}\Delta_{k,n}+ m(1-t)\sum_{l=1}^{n-1}\Delta_{m,l} + \Delta_{m,n}\max\{n(1-t), m(1-t)\} \\
        = \quad & n(1-t)\left(\frac1n - \Delta_{m,n}\right) + m(1-t)\left(\frac1m - \Delta_{m,n}\right) + \Delta_{m,n}\max\{n(1-t), m(1-t)\} \\
        = \quad & (1-t)\left(2 - (m\wedge n)\Delta_{m,n}\right)
    .\end{align*}
    This yields
    \[
        \lambda_U(C_{\nearrow}) 
        = 2 - \lim_{t\nearrow 1} \frac{1 - C_{\nearrow}(t,t)}{1 - t}
        = 2 - \lim_{t\nearrow 1} \frac{C_{\nearrow}(1,1) - C_{\nearrow}(t,t)}{1 - t}
        = \Delta_{m,n}(m\wedge n)
    ,\]
    finishing the proof.
\end{proof}

\begin{proof}[Proof of Corollary \ref{cor:xi_checkerboard}]
    Note that
    \[
        \sum_{i,j=1}^{m,n} \Delta_{i,j}^2
        \leq \sum_{i=1}^m \left(\sum_{j=1}^n \Delta_{i,j}\right)^2
        = \sum_{i=1}^m\frac1{m^2}
        = \frac{1}{m}
    \]
    and in the same way
    \[
        \sum_{i,j=1}^{m,n} \Delta_{i,j}^2 \leq \frac{1}{n}
    .\]
    Hence, by Proposition \ref{lem:checkerboard_measures_of_association} \ref{itm:xi_checkerboard}, we have
    \[
        \abs{\xi\left(C^\Delta_{\nearrow}\right) 
        - \xi\left(C^\Delta_{\Pi}\right)}
        = \frac{m\tr(\Delta^\top \Delta)}{n}
        = \frac{m}{n}\sum_{i,j=1}^{m,n} \Delta_{i,j}^2
        \le \frac{m}{n(m\vee n)}
        = \begin{cases}
            \frac{m}{n^2}, & \text{if } m\le n \\
            \frac{1}{n}, & \text{if } m > n
        \end{cases}
    ,\]
    as claimed.
\end{proof}

\bibliography{Literature}{}
\bibliographystyle{plain}

\end{document}